\documentclass[journal]{IEEEtran}

\usepackage{amsmath}    
\usepackage{graphics,epsfig,subfigure}    
\usepackage{verbatim}   
\usepackage{color}      
\usepackage{subfigure}  
\usepackage{hyperref}   
\usepackage{amssymb}
\usepackage{epstopdf}
\usepackage{graphicx}
\usepackage{mathrsfs} 

\newcommand{\expect}[1]{\mathbb{E}\bigg[#1\bigg]}
\newcommand{\expectm}[1]{\mathbb{E}\big[#1\big]}

\newcommand{\bv}[1]{{\boldsymbol{#1} }}
\newcommand{\script}[1]{{{\cal{#1} }}}

\newcommand{\tsf}[1]{\textsf{#1}}

\newtheorem{fact}{\textbf{Fact}}

\newtheorem{lemma}{\textbf{Lemma}}

\newtheorem{theorem}{\textbf{Theorem}}

\newtheorem{definition}{\textbf{Definition}}

\allowdisplaybreaks
\begin{document}

\title{A Benes Packet Network}
\author{\large{Longbo Huang, Jean Walrand}  
\thanks{Longbo Huang    (http://www.eecs.berkeley.edu/$\sim$huang) and Jean Walrand (http://www.eecs.berkeley.edu/$\sim$wlr)  are with the EECS department of University of California, Berkeley, CA, 94720.  
} 
\markboth{Draft}{Huang}
}

\maketitle

\begin{abstract}
Benes networks are constructed with simple switch modules and have many advantages, including small latency and requiring only an almost linear number of switch modules. As circuit-switches, Benes networks are rearrangeably non-blocking, which implies that they are full-throughput as packet switches, with suitable routing.

Routing in Benes networks can be done by time-sharing permutations. However, this approach requires centralized control of the switch modules and statistical knowledge of the traffic arrivals. We propose a backpressure-based routing scheme for Benes networks, combined with end-to-end congestion control.  This approach achieves the maximal utility of the network and requires only four queues per module, independently of the size of the network.

\end{abstract}

\section{Introduction}
Data centers have gradually become one of our most important computing resources. For instance, search engines, web emails such as Gmail and Hotmail, social network websites such as Facebook, and data processing applications such as Hadoop are provided by data centers. 
Consequently, the networking of servers and resource allocation in data centers have become important problems.

We develop a networking solution, a \emph{Benes packet network}, which consists of a Benes architecture, a flow utility maximization mechanism, and a backpressure-based scheduling algorithm. 
Specifically, we propose interconnecting the data center servers using a Benes network built with simple commodity switch modules. We formulate the resource allocation objective as a network flow utility maximization problem to guarantee a fair share of the network resources. Lastly, we develop a low-complexity backpressure-based scheduling algorithm, called Grouped-Backpressure (\tsf{G-BP}), to achieve the optimal system performance. 
The \tsf{G-BP} algorithm is provably optimal and automatically handles changing traffic. Our approach only requires each switch module to maintain \emph{four} queues, independently of  the network size, and hence can easily be implemented in practice. 


Many papers explore networking solutions for data centers. 
\cite{jellyfish-12} proposes using a random graph based approach to enable incremental network growth for data centers. \cite{vl2-sigcomm09} proposes a  network architecture based on Clos network and random traffic splitting. \cite{dcell-sigcomm08} develops a hierarchical network structure for data centers. \cite{scafida-sigcomm10} uses the preferential attachment approach to design network topologies for data centers. \cite{scalable-sigcomm08} proposes a fat-tree based network architecture.  
\cite{camdoop-nsdi-12} develops a MapReduce-like system based on a cube-like  architecture to exploit the in-network aggregation possibilities. \cite{cthrough-sigcomm10} designs optical networks for data centers. 
However, we note that the aforementioned works mostly focus on designing the network architecture and achieving uniform load balancing. Hence, the proposed  solutions do not immediately apply to problems where different flows have different service requirements. Moreover, the solutions developed in the above works lack system performance guarantees. 

In this work, we aim at obtaining a network solution that combines practicality, generality, provable optimality, and low complexity. Specifically, we propose interconnecting the data center servers by a Benes network. 
As circuit-switches, Benes networks are known to be rearrangeably non-blocking and  can easily be built with only an almost linear number of simple switch modules in the network size \cite{benes-net}, \cite{benes-net-book}. Thus, adopting the Benes network architecture not only guarantees high system throughput and low end-to-end packet delay (if routing and scheduling are done properly), but also eliminates the need for employing expensive switch devices whose cost does not scale easily as the data center size increases. 
Under the Benes network architecture, we establish a mathematical formulation for determining the allocation of  network resources to cope with the heterogeneity of the data traffic service requirements. Our formulation leverages the network utility maximization framework \cite{Kellyelastic97}, \cite{layering-chiang07}, which has been proven to be a general mechanism for handling network resource allocation problems. 
%

Finally, to reap the full benefits of the Benes network architecture and the resource allocation framework in a practical manner, we develop a routing and scheduling algorithm that has provable system performance guarantees and  a very low implementation complexity. Our algorithm is constructed based on the recently developed backpressure network optimization technique \cite{neelynowbook}, combined with an end-to-end congestion control mechanism. However, different from previous backpressure algorithms, e.g., \cite{neelyfairness}, \cite{ying_clusterbp_infocom08}, \cite{buisrikant_infocom09}, which either require that  the number of queues each switch module has to maintain is proportional to the network size, or only apply to problems with single-path routing, our algorithm uses a novel \emph{traffic grouping} idea and allows us to use only \emph{four} queues per switch module regardless of the network size. Moreover, our algorithm automatically explores all the possible routes to fully utilize network capacity. These distinct  features make our algorithm very suitable for practical implementation. 
%


This paper is organized as follows. In Section \ref{section:model}, we present the  system model and state our objective. In Section \ref{section:labeling}, we set up the notations. Then, we explain the intuition of our design approach and describe all the needed components of the Group-Backpressure (\tsf{G-BP}) algorithm in Section \ref{section:intuition}. We present the \tsf{G-BP} algorithm and analyze its performance in Section \ref{section:gbp}. 
Simulation results are presented in Section \ref{section:simulation}. We conclude the paper in Section \ref{section:conclusion}.

\vspace{-.1in}
\section{System model}\label{section:model}
We consider the system shown in Fig. \ref{fig:benes-16x16-phy}, where a Benes network connects a set of communicating servers. In this system, each rectangle is a \emph{switch module} having two input and two output links. Each link has a capacity of $1$ packet/slot. The smaller nodes are the \emph{servers}. 
Traffic flows are generated from the servers on the left, called \emph{input servers}, and are going to the servers on the right, called \emph{output servers}. \footnote{It is straightforward to extend our results to include bi-directional traffic flows.} 
We assume that the system operates in slotted time, i.e., $t\in\{0, 1, 2, ...\}$.  The discrete time assumption is for convenience of the analysis.  The actual network would operate in an asynchronous way with variable length packets.
%

\begin{figure}[cht]
  \centering
  \vspace{-.1in}
\includegraphics[height=1.8in, width=3.0in]{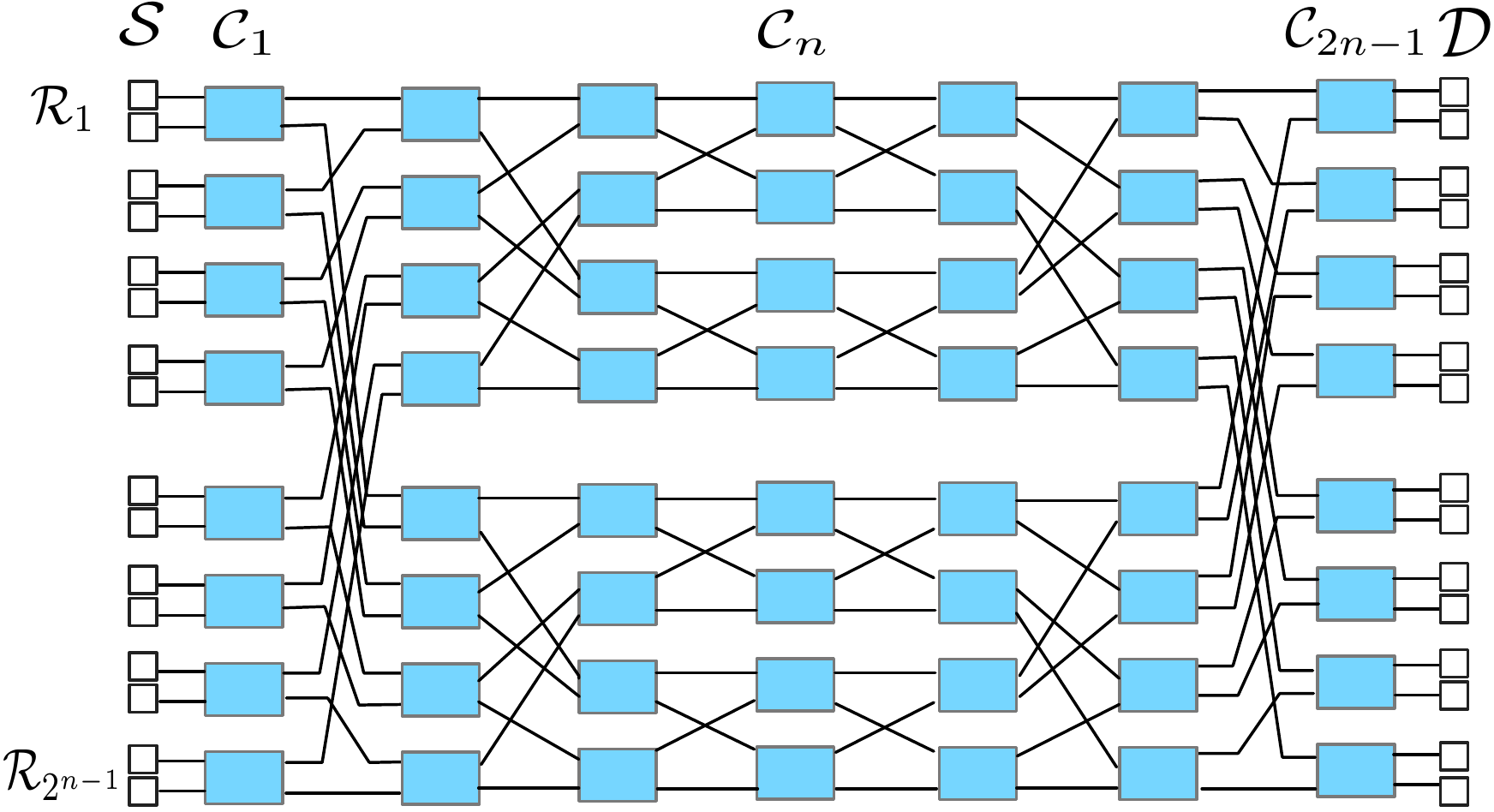}
  \vspace{-.1in}
  \caption{A $16\times16$ Benes network connecting $16$ input servers ${\cal S}$   to $16$ output servers ${\cal D}$. The rectangles are the switch modules that form the Benes network. $\script{R}_i$ refers to row $i$ of the Benes network and $\script{C}_j$ refers to column $j$. } 
  \label{fig:benes-16x16-phy}
  \vspace{-.1in}
\end{figure}

\vspace{-.1in}
\subsection{Admission control and flow utility}
We label the flows according to their source and destination servers.
Specifically, we call the traffic entering from input server $s$ and going to output server $d$  the $(s, d)$ flow. 
We use $A_{sd}(t)$ to denote the number of $(s, d)$ packets generated at input server $s$ at time $t$. 
We assume that for every $(s,d)$ flow, the random variables $\{A_{sd}(t), t = 0, 1, \ldots\}$ are i.i.d. and have mean $\lambda_{sd}=\expectm{A_{sd}(t)}$. Our results can be extended to incorporate much more general arrival processes, e.g., Makov-modulated arrivals. 
We also assume that there exists some finite constant $A_{\max}$ such that $0\leq A_{sd}(t)\leq A_{\max}$ for all $(s, d)$ and for all time $t$. 

In every time slot $t$, each input server performs admission control to determine how many packets to inject into the network.  We denote $0\leq R_{sd}(t)\leq A_{sd}(t)$ the number of $(s, d)$ flow packets \emph{actually}  admitted by input server $s$ for transmission at time $t$.  We then denote the average rate of the $(s, d)$ flow packets by $r_{sd}$, defined as: \footnote{Throughout this paper, we assume that all the  limits exist.}
\begin{eqnarray}
r_{sd} \triangleq \lim_{T\rightarrow\infty}\frac{1}{T}\sum_{t=0}^{T-1}\expectm{R_{sd}(t)}. \label{eq:avg-rate}
\end{eqnarray}
Each $(s, d)$ flow is associated with a utility function $U_{sd}(r_{sd})$, which is concave increasing in its average rate $r_{sd}$. We assume that the utility functions have finite first derivatives and denote $\beta$ their maximum value, i.e., 
\begin{eqnarray}
\beta\triangleq\max_{sd}U'_{sd}(0). \label{eq:beta-def}
\end{eqnarray}

\vspace{-.1in}
\subsection{Stability and objective} 
In this paper, we say that a queue with queue size process $\{Q(t)\geq0, t=0, 1, 2,...\}$ is stable if:
\begin{eqnarray}
\limsup_{T\rightarrow\infty} \frac{1}{T}\sum_{t=0}^{T-1}\expectm{Q(t)}<\infty.
\end{eqnarray}
Then,  we say that a network is  stable if all the queues in the network are stable, and call a routing and scheduling policy that ensures network stability a stabilizing policy. 
We use $\Lambda_n$ to denote the capacity region of a $2^n\times2^n$ Benes network, being the set of arrival vectors under which there exist stabilizing routing and scheduling policies. 

Depending on the routing and scheduling algorithm, the network queueing structure can be quite different. 
Our objective is to find a low-implementation-complexity stabilizing routing and scheduling policy that  maximizes the aggregate flow utility of the network, i.e., 
\begin{eqnarray}
\max: &&U(\bv{r}) \triangleq \sum_{s, d} U_{sd}(r_{sd}) \label{eq:utility-def}\\
\text{s.t.} && \bv{r}\in\Lambda_n, \nonumber
\end{eqnarray}
where $\bv{r}=(r_{sd}, \,\forall\, (s, d))$ with $r_{sd}$ being the average rate of the $(s, d)$ flow defined in (\ref{eq:avg-rate}). 
We denote by $\bv{r}^{\textsf{opt}}$ the rate vector that achieves the optimal utility over  all stabilizing policies. 

Note that our formulation (\ref{eq:utility-def}) is indeed very general. The heterogeneity of traffic flow service requirements can easily be taken into account by designing appropriate utility functions. Also note that, although our system model is similar to those in \cite{neelyfairness}, in our paper, the queueing structure is also part of the algorithm design problem.

\vspace{-.1in}
\subsection{Discussion}
The problem of optimal routing and scheduling in a Benes network can be solved by using the well-known backpressure routing algorithm \cite{neelynowbook}. However, this approach requires each node to maintain a separate queue for each output server. Thus, each node has to maintain $2^{n}$ queues, which is not practical  when the size of the Benes network (number of servers) increases. 
Recent works \cite{ying_clusterbp_infocom08} and \cite{buisrikant_infocom09} propose backpressure-based algorithms that use much fewer queues. However, the algorithm in  \cite{ying_clusterbp_infocom08} requires nodes to maintain a separate queue for each cluster of the network nodes and needs a pre-defined clustering algorithm, whereas the method in \cite{buisrikant_infocom09} is designed for single-path routing. 
Below, we develop a novel low-complexity approach called Grouped-Backpressure (\tsf{G-BP}). Our approach allows us to use only \emph{four} queues per node regardless of the network size. 

\vspace{-.1in}
\section{Benes network structure and labeling}\label{section:labeling}
In this section, we explain the structure of Benes networks  and set up our notations. 

\vspace{-.1in}
\subsection{Benes network construction}\label{section:benes-connect}
We first explain how a $2^n\times2^n$ Benes network is constructed \cite{benes-net} \cite{benes-net-book}: 
Start with a basic $2\times2$ Benes network as in Fig. \ref{fig:benes-2x2}. Then, construct a $2^n\times2^n$ Benes network as follows: 

{(Step I)-{Concatenation:}} Vertically concatenate two $2^{n-1}\times2^{n-1}$ Benes networks. Call them the \emph{upper subnetwork} and the \emph{lower subnetwork}, e.g., $m_3$ and $m_4$ in Fig. \ref{fig:benes-4x4}. 
Then, horizontally place two columns of $2^{n-1}$ basic $2\times2$ modules, one on each side of the concatenated subnetworks. Call the modules on the left of the concatenated subnetworks the \emph{input switch modules}, e.g., $m_1$ and $m_2$, and the modules on the right the \emph{output switch modules}, e.g., $m_5$ and $m_6$. 

{(Step II)-{Connect input modules:}} Connect the upper output link of the input module in row $k$ to the $k^{\text{th}}$ input link of the upper subnetwork, and connect its  lower output link to the $k^{\text{th}}$ input link of the lower subnetwork. 

{(Step III)-{Connect output modules:}} Connect the $k^{\text{th}}$ output link of the upper subnetwork to the upper input link of the $k^{\text{th}}$ output module, i.e., the output module in row $k$, and connect the $k^{\text{th}}$ output link of the lower subnetwork to the lower input link of the $k^{\text{th}}$ output module. 
\begin{figure}[cht]
  \centering
    \vspace{-.1in}
    \subfigure[A basic $2\times2$ Benes network]{\label{fig:benes-2x2}\includegraphics[height=0.4in, width=1.0in]{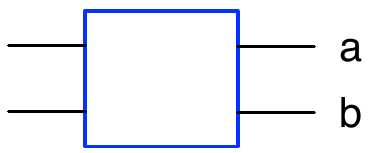}}
  \subfigure[A $4\times4$ Benes network]{\label{fig:benes-4x4}\includegraphics[height=1.0in, width=2.5in]{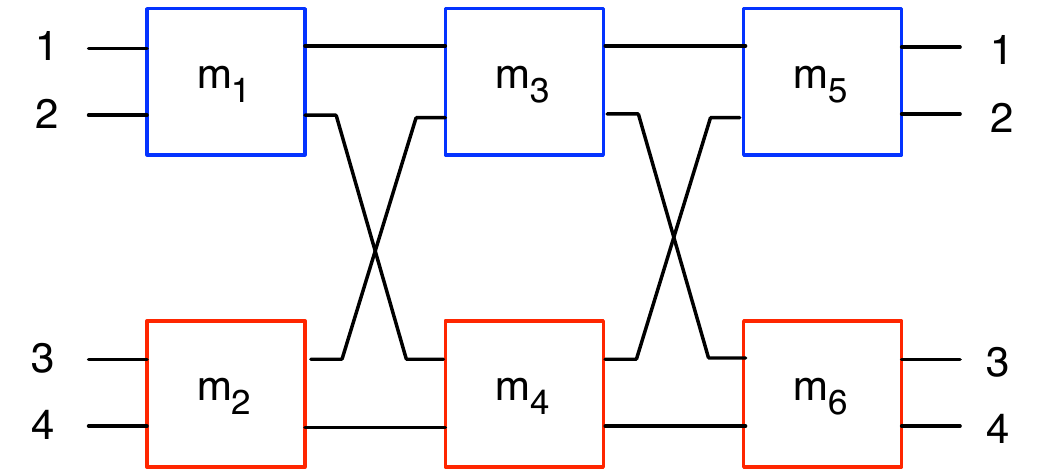}}
  \subfigure[A general $2^n\times2^n$ Benes network]{\label{fig:benes-general}\includegraphics[height=1.0in, width=2.5in]{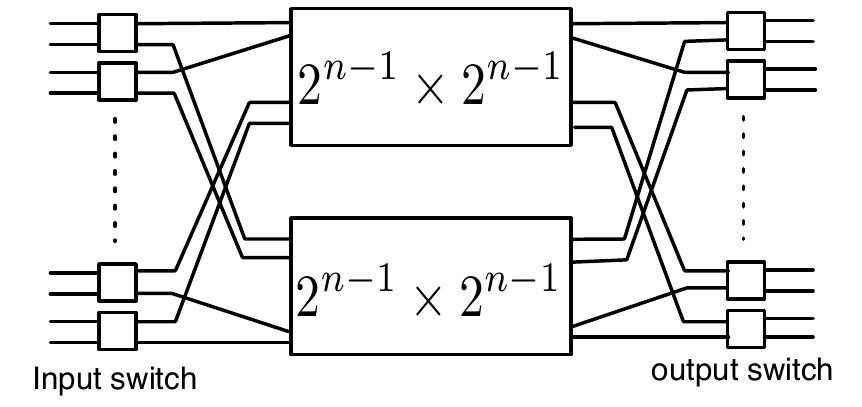}}
  \caption{The structure of Benes networks.}
  \label{fig:benes-demo}
    \vspace{-.1in}
\end{figure}

\vspace{-.1in}
\subsection{Labeling a Benes network with servers} 
We first specify how we label  a $2^n\times2^n$ Benes network. 
We denote $\mathbb{B}_n$  the $2^n\times2^n$ Benes network (excluding the input and output servers). Then, we divide the Benes network into rows and columns. In a $2^n\times 2^n$ Benes network,  there are $2^{n-1}$ rows, denoted by $\{\script{R}_i, i=1, ..., 2^{n-1}\}$. We then denote the $2n-1$ columns by $\{\script{C}_j, j=1, ..., 2n-1\}$. 
For any node $m$ in the Benes network, we use $i_m$ and $j_m$ to denote its row number and column number. 
For the input and output servers connecting to the Benes network, we label them using their row numbers. The set of input servers are denoted by $\script{S}=\{1, 2, ...,2^n\}$ and the set of output servers are  denoted by $\script{D}=\{1, 2, ...,2^n\}$. 
Note that both $\script{S}$ and $\script{D}$ have $2^n$ rows (the small squares in Fig. \ref{fig:benes-16x16-phy}). 
As in Section \ref{section:benes-connect}, we call the nodes in $\script{C}_1$ the \emph{input switch} modules and the nodes in $\script{C}_{2n-1}$ the \emph{output switch} modules. 

From the construction rules of Benes networks and the way the servers are connected to the Benes network, we see that for every node $m\in\mathbb{B}_n$, there are two nodes in column $\script{C}_{j_m+1}$ to which it connects (for a node $m\in\script{C}_{2n-1}$, it connects to two nodes in $\script{D}$). 
We denote the node with a smaller row number by $m_{u}$ and the other one by $m_{l}$. 
There are also two nodes in column $\script{C}_{j_m-1}$ that connect  to $m$ (if $m\in\script{C}_1$, there are two nodes in $\script{S}$ connecting to it). We denote these two nodes by $\script{M}_m$. Among these nodes, those that have $m$ as their next hop with a smaller row number are denoted by $\script{M}_m^u$, and the other nodes having $m$ as their next hop node with a larger row number are denoted by $\script{M}_{m}^l$, i.e., 
\begin{eqnarray*}
&&\script{M}_m^u = \{m' \in\script{C}_{j_m-1}\,\,|\,\, m'_u=m\}, \\
&&\script{M}_m^l = \{m' \in\script{C}_{j_m-1}\,\,|\,\, m'_l=m\}. 
\end{eqnarray*}
For $m\in\script{C}_1$, we simply use $\script{M}_m$ to denote the input servers  that connect to it. For each input server $s\in\script{S}$, we use $m(s)$ to denote the node in $\script{C}_1$ it connects to. 
We call the servers in rows $1$ to $2^{n-1}$ the \emph{upper division} servers, and call all  the other servers the \emph{lower division} servers. 
We then call a flow whose destination is an upper division server an \emph{upper division flow}. Otherwise it is a \emph{lower division flow}.  

For a $2^n\times2^n$ Benes network $\mathbb{B}_n$, we define the nodes in $\script{C}_n$ as the \emph{partition} nodes. 
%
From the construction rules of the Benes networks, we first have the following observation: 
\begin{fact}\label{fact:partition-node-coincide}
For a $2^n\times2^n$ Benes network, its partition nodes coincide with the partition nodes of its two $2^{n-1}\times2^{n-1}$ subnetworks. 
$\Diamond$
\end{fact}

Below, we denote the upper outgoing link of a switch module by link $a$ and the lower outgoing link by link $b$ (see Fig. \ref{fig:benes-2x2}). We use $\script{O}_m^a$ to denote the set of output servers that can be reached by traversing the upper outgoing link $a$ of node $m$ and use $\script{O}_m^b$ to denote the set of output servers that can be reached by traversing  link $b$. 
We then have the following simple lemma, which can be seen from the construction rules of Benes networks. 
\begin{lemma}\label{lemma:unique-path}
(a) Starting from any partition node $m\in\script{C}_n$,   there is a unique path to any  output server $d\in\script{D}$. (b) For every node $m\in\script{C}_{n+l}, l\geq0$, we have: 
\begin{eqnarray}
\hspace{-.3in}&&\script{O}_m^a=\{ \kappa_m2^{n-l}+1, ...,   (\kappa_m+\frac{1}{2})2^{n-l}\}, \label{eq:output-set-a}\\
\hspace{-.3in}&&\script{O}_m^b=\{ (\kappa_m+\frac{1}{2})2^{n-l}+1, ...,   (\kappa_m+1)2^{n-l}\}, \label{eq:output-set-b}
\end{eqnarray}
where $\kappa_m\triangleq (i_m-1) \mod 2^l$. 
$\Diamond$
\end{lemma}

 \vspace{-.1in}
\section{Intuition and key components of Grouped-Backpressure}\label{section:intuition}
In this section, we present the idea and all the needed components for our Grouped-Backpressure algorithm (\tsf{G-BP}), which will be used to achieve the optimal flow utility under the Benes network architecture. 

\vspace{-.1in}
\subsection{The idea}
The idea of Grouped-Backpressure is to ``group'' all the flows into two mixed flows, the upper division flow and the lower division flow. Then, we construct a scheme for routing the \emph{mixed} traffic in the first half of the network based on a fictitious reference system.  This approach allows us to use very few queues per node. 
However, due to this traffic mixing, we lose the ability to control each individual flow inside the network. Hence, the flows can be routed arbitrarily inside the network, in which case certain nodes may receive more  traffic than they can handle and become unstable. 
In order to resolve this problem, we impose a special queueing structure at each node to ensure that routing and scheduling is done in a fully symmetric manner. 
With this approach, we guarantee that every flow is split into sub-flows with equal rates and routed through the partition nodes. 
In the second half of the network, by Lemma \ref{lemma:unique-path}, each packet will traverse a unique path to its destination. Hence, we will do a ``free-flow'' routing. 
Using the symmetric structure of the Benes network, we then show that the \tsf{G-BP} algorithm can stabilize the network and achieve maximum utility. Our approach is demonstrated in Fig. \ref{fig:gbp-demo}. 
\begin{figure}[cht]
  \centering
  \vspace{-.1in}
\includegraphics[height=1.8in, width=3.0in]{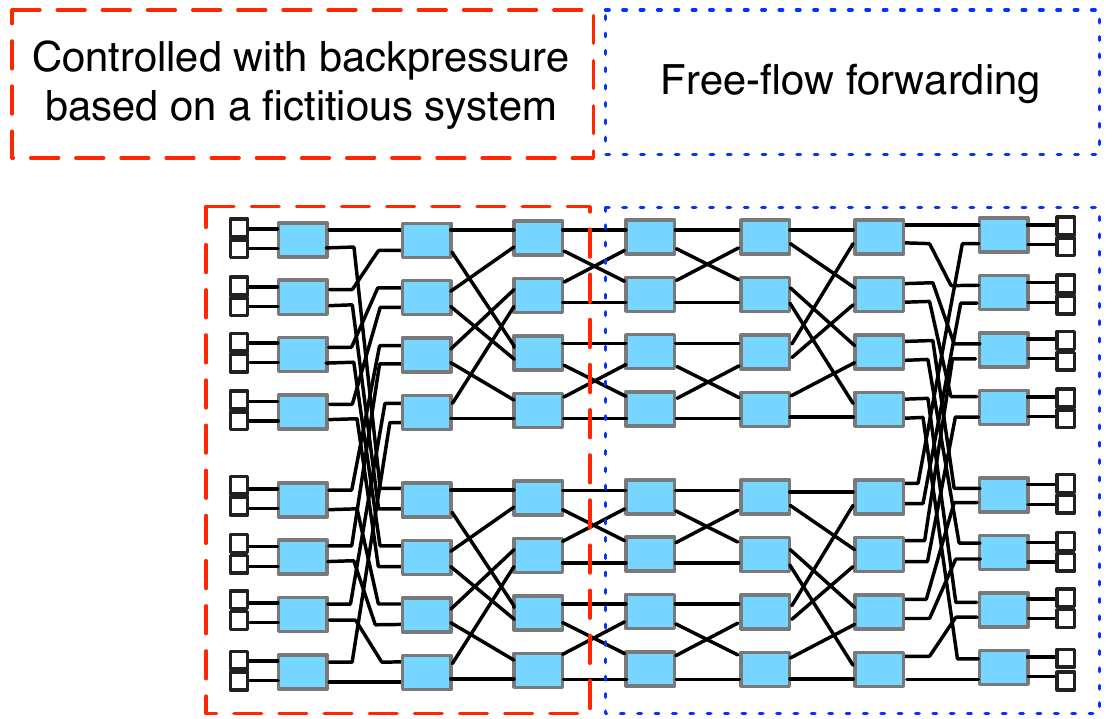}
  \vspace{-.1in}
  \caption{The pictorial illustration of the \tsf{G-BP} algorithm. The first half of the network is controlled by a backpressure-like algorithm based on a fictitious reference system. The second half of the network uses a ``free-flow'' scheme for packet delivery.}
  \label{fig:gbp-demo}
    \vspace{-.1in}
\end{figure}

\vspace{-.1in}
\subsection{A fictitious reference system}
In order to guide the routing and scheduling of the grouped traffic, we create a fictitious reference system as follows. 
\begin{enumerate}
\item Remove all the nodes in columns $n+1$ to $2n-1$.
\item Create two fictitious destination nodes $D_1$ and $D_2$, where $D_1$ represents the common destination for the upper division flows and $D_2$ represents the common destination for the lower division flows. 
\item Connect each partition node, i.e., a node in $\script{C}_n$, to $D_1$ with a link of capacity $1$ packet/slot and to  $D_2$ with a link of capacity $1$ packet/slot. 
\end{enumerate}
An example of the fictitious system is shown in Fig. \ref{fig:benes-16x16-fic} for a $16\times16$ Benes network. 
The  fictitious system will be used as a reference system to guide us on serving the grouped traffic. Specifically, we will design a backpressure-based algorithm for the fictitious system, and use the \emph{exact} same actions to control the nodes in columns $1$ to $n-1$ in the physical system. 
This approach has the useful property that it allows us to use only  $4$ queues per node.  
%

%
\begin{figure}[cht]
  \centering
    \vspace{-.1in}
\includegraphics[height=1.8in, width=3.0in]{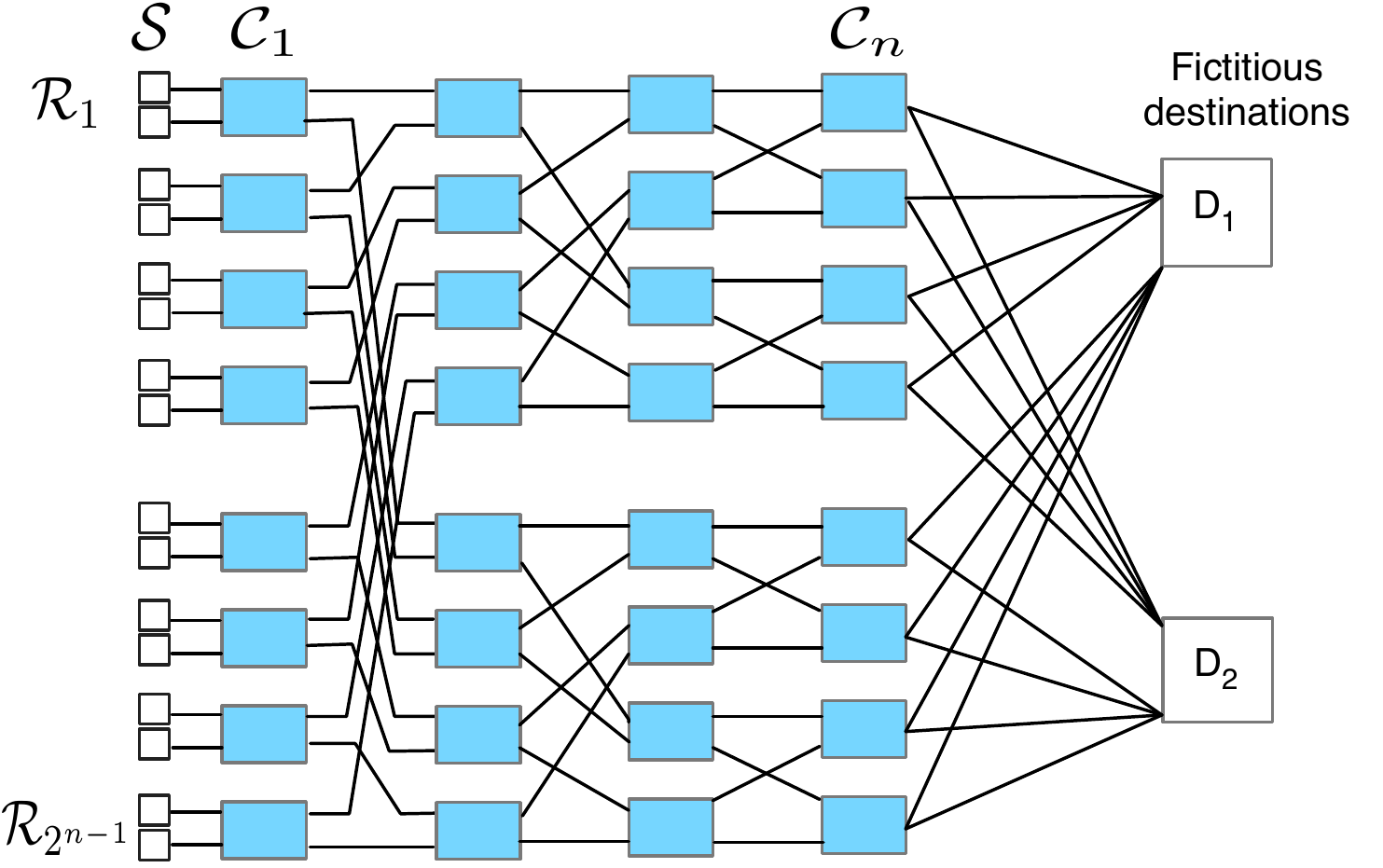}
  \vspace{-.1in}
  \caption{The fictitious reference system for the $16\times16$ Benes network with servers.}  \label{fig:benes-16x16-fic}
    \vspace{-.1in}
\end{figure}

\vspace{-.1in}
\subsection{Queue structure and load balancing}
Since in the reference system we only have $2$ destinations and do not distinguish flows inside the network, if routing is not done carefully, it can happen that most of the traffic going to an  output port is routed to a single partition node and causes instability of the node. In order to resolve this issue, we  impose a special queueing structure on the switch nodes to balance all the traffic, so that each flow is equally split among all possible paths and routed to  the partition nodes. Doing so, we guarantee that as long as the traffic rate is supportable (will be explained later), no node will be overwhelmed. 

We now specify our queueing structure for both the fictitious system and the physical system:  
\subsubsection{Input servers in both systems} For each input server $s\in\script{S}$, we maintain $2$ queues per node as follows: 
\begin{itemize}
\item $Q_s^{\tsf{U}}(t)$: number of \emph{upper division} flow packets stored at input server $s$;
\item $Q_s^{\tsf{L}}(t)$: number of \emph{lower division} flow packets stored at input server $s$. 
\end{itemize}
These two queues evolve according to the following dynamics: 
\begin{eqnarray}
\hspace{-.35in}&&Q_{s}^{\script{T}}(t+1) = \big[Q_{s}^{\script{T}}(t) - \mu_{s, m(s)}^{\script{T}}(t)  \big]^+   + R_{s}^{\script{T}}(t). 
\label{eq:queue-source-q}
\end{eqnarray}
Here the notation $[x]^+=\max[x, 0]$ and the notation $\script{T}\in\Omega_s \triangleq\{\script{\tsf{U}, \tsf{L}}\}$ denotes the ``type'' of the traffic at the input servers, and $R_{s}^{\script{T}}(t)$ denotes the aggregate arrival to  $Q_s^{\script{T}}(t)$, i.e., 
\begin{eqnarray}
\hspace{-.7in}&&R_{s}^{\tsf{U}}(t) = \sum_{d\leq 2^{n-1}}R_{sd}(t), \,\, R_{s}^{\tsf{L}}(t) = \sum_{ d>2^{n-1}}R_{sd}(t). \label{eq:input-server-rate}
\end{eqnarray}

\subsubsection{Switch modules in columns $1$ to $n-1$ in both systems} We maintain $4$ queues per node as follows: 
\begin{itemize} 
\item $Q_{m}^{\textsf{UU}}(t)$: number of \emph{upper division} flow packets that will be routed through $m_u$; 
\item $Q_{m}^{\textsf{UL}}(t)$: number of \emph{upper division} flow packets that will be routed through $m_l$; 
\item $Q_{m}^{\textsf{LU}}(t)$: number of \emph{lower division} flow packets that will be routed through $m_u$; 
\item $Q_{m}^{\textsf{LL}}(t)$: number of \emph{lower division} flow packets that will be routed through $m_l$. 
\end{itemize} 
Now define $\Omega_B\triangleq\{\tsf{UU}, \tsf{UL}, \tsf{LU}, \tsf{LL}\}$ and use $\script{T}\in\Omega_B$ to denote the type of these queues at the switch nodes. We see that the queues evolve  according to the following dynamics: 
\begin{eqnarray}
\hspace{-.35in}&&Q_{m}^{\script{T}}(t+1) \label{eq:queue-bernes-net}\\
\hspace{-.35in}&&\qquad\quad \leq \big[Q_{m}^{\script{T}}(t) - \mu_{m, m(\script{T})}^{\script{T}}(t)  \big]^+   + R_{m}^{\script{T}}(t), \,\forall\,\script{T}\in\Omega_B. \nonumber 
\end{eqnarray}
Here $m(\script{T})$ is the next hop node corresponding to the type $\script{T}$ traffic, i.e., $m(\script{T})=m_u$ for $\script{T}\in\{\tsf{UU}, \tsf{LU }\}$ and $m(\script{T})=m_l$ otherwise. 
$R_{m}^{\script{T}}(t)$ is the \emph{aggregate} arrivals to $Q_{m}^{\script{T}}(t)$, given by: 
\begin{eqnarray}
\hspace{-.5in}&&R_{m}^{\tsf{UU}}(t) = X_{m}(t)R_{m}^{\tsf{U}}(t),  \,\,
R_{m}^{\tsf{UL}}(t) =(1-X_{m}(t))R_{m}^{\tsf{U}}(t),\\
\hspace{-.5in}&&R_{m}^{\tsf{LU}}(t) = Y_{m}(t)R_{m}^{\tsf{L}}(t), \,\,\,
R_{m}^{\tsf{LL}}(t) =(1-Y_{m}(t))R_{m}^{\tsf{L}}(t), 
\end{eqnarray}
where $R_{m}^{\tsf{U}}(t)$ and $R_{m}^{\tsf{L}}(t)$ are the aggregate upper and lower division arrivals to node $m$, given by: 
\begin{eqnarray}
\hspace{-.6in}&&R_{m}^{\tsf{U}}(t) = \sum_{m'\in\script{M}_{m}^u}\mu^{\textsf{UU}}_{m',m}(t)  +  \sum_{m'\in\script{M}_{m}^l} \mu^{\textsf{UL}}_{m', m}(t),  \label{eq:income-rate-split-1}\\
\hspace{-.6in}&&R_{m}^{\tsf{L}}(t) =\sum_{m'\in\script{M}_{m}^u}\mu^{\textsf{LU}}_{m', m}(t)  +  \sum_{m'\in\script{M}_{m}^l} \mu^{\textsf{LL}}_{m', m}(t). \label{eq:income-rate-split-2}
\end{eqnarray}
The variables $X_m(t)$ and $Y_m(t)$ are i.i.d. Bernoulli variables taking values $0$ or $1$ with equal probabilities,  introduced for ensuring an equal division of the flow rates. 
Note that we have used inequality in (\ref{eq:queue-bernes-net}). This is because the actual packet arrivals to $Q_m^{\script{T}}(t)$ may be less than $R_{m}^{\script{T}}(t)$ as the upstream nodes may not have enough packets to fulfill the allocated transmission rates.  
%
Our queueing structure and traffic splitting scheme are demonstrated in Fig. \ref{fig:queue-demo}. 
\begin{figure}[cht]
  \centering
  \vspace{-.1in}
\includegraphics[height=1.4in, width=3.2in]{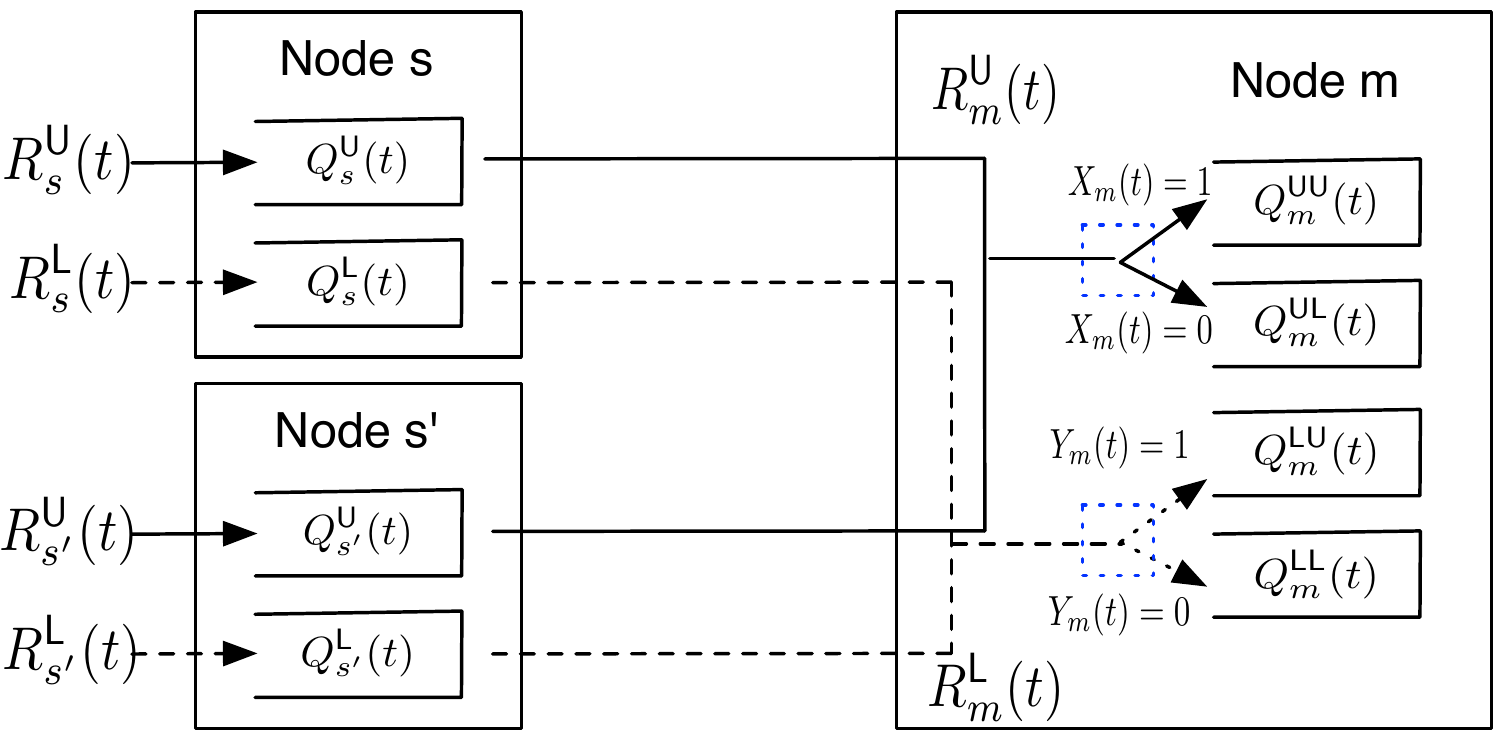}
  \vspace{-.1in}
  \caption{The  queueing structure and traffic splitting method.}
  \label{fig:queue-demo}
    \vspace{-.1in}
\end{figure}

\subsubsection{Partition nodes in the fictitious system} Each node $m\in\script{C}_n$ maintains only two queues $Q_m^{D_1}(t)$ and $Q_m^{D_2}(t)$ with the following dynamics: 
\begin{eqnarray}
\hspace{-.3in}&&Q_{m}^{D_i}(t+1) \leq \big[Q_{m}^{D_i}(t) - \mu_{m, D_i}(t)  \big]^+   +  
R_{m}^{D_i}(t). \label{eq:queue-partition-1}
\end{eqnarray}
Here $R_{m}^{D_1}(t)=R_m^{\tsf{U}}(t)$ and $R_{m}^{D_2}(t)=R_m^{\tsf{L}}(t)$ are the aggregate arrivals defined in (\ref{eq:income-rate-split-1}) and  (\ref{eq:income-rate-split-2}).

\subsubsection{Nodes in columns $n$ to $2n-1$ in the physical system} Each node $m\in\cup_{j=n}^{2n-1}\script{C}_j$ maintains two First-In-First-Out (FIFO) queues $Q^a_{m}(t)$ and $Q^b_{m}(t)$, one for the upper output link $a$ and the other for the lower output link $b$ (see Fig. \ref{fig:benes-2x2}). The arrivals are placed into the queues according to their destinations, i.e.,
\begin{eqnarray}
\hspace{-.5in}&&Q^a_{m}(t+1) = \big[  Q^a_{m}(t) - \mu_{m, m_u}(t)  \big]^+ + \sum_{m'\in\script{M}_{m}}\mu^{a}_{m', m}(t),  \label{eq:queue-dyn-phy-1}\\
 \hspace{-.5in}&&Q^b_{m}(t+1) = \big[  Q^b_{m}(t) - \mu_{m, m_l}(t) \big]^+ + \sum_{m'\in\script{M}_{m}}\mu^{b}_{m', m}(t).  \label{eq:queue-dyn-phy-2}
\end{eqnarray}
Here $\mu^{a}_{m', m}(t)=\sum_{s}\sum_{d\in\script{O}_{m}^a}\tilde{\mu}_{m',m}^{sd}(t)$, where $\tilde{\mu}_{m',m}^{sd}(t)$ denotes the \emph{actual} number of flow $(s, d)$ packets sent from $m'$ to $m$ at time $t$, and $\mu^{b}_{m', m}(t)=\sum_{s}\sum_{d\in\script{O}_{m}^b}\tilde{\mu}_{m',m}^{sd}(t)$ denotes the number of packets that need to traverse the lower outgoing link $b$ to their destinations.  

Notice that in both systems, each partition node only maintains two queues and does not further split the traffic. This is because in the fictitious system, the next hop nodes of a partition node are $D_1$ and $D_2$, whereas in the physical system, the flow $(s, d)$ packets at the partition nodes will be delivered  to output server $d$ following a unique path according to Lemma \ref{lemma:unique-path}. 

We now show that under the special structure of the Benes network,  our queueing structure and traffic splitting scheme generate a balanced routing across the network. This is summarized in the following lemma, where, for $\script{T} \in \Omega_B$, we use $\overline{\mu}_{m, m(\script{T})}^{\script{T}}$ to denote the 
time average transmission rates of the type $\script{T}$ traffic from $m$ to $m(\script{T})$. 
Specifically, 
\begin{eqnarray*}
\hspace{-.3in}&&\overline{\mu}_{m, m(\script{T})}^{\script{T}} \triangleq\lim_{T\rightarrow\infty}\frac{1}{T}\sum_{t=0}^{T-1}\expectm{\tilde{\mu}_{m, m(\script{T})}^{\script{T}}(t)}. 
\end{eqnarray*} 
Here $\tilde{\mu}_{m, m(\script{T})}^{\script{T}}(t)$ denotes the \emph{actual} number of type $\script{T}$ packets sent over the link $[m, m(\script{T})]$ at time $t$. 
Similarly, we use $\overline{\mu}^{sd}_{m}$ to denote the average rate of the flow $(s, d)$ packets going through a node $m$, i.e., 
\begin{eqnarray*}
\overline{\mu}^{sd}_{m}\triangleq\lim_{T\rightarrow\infty}\frac{1}{T}\sum_{t=0}^{T-1}\expectm{\tilde{\mu}_{m, m_u}^{sd}(t)+\tilde{\mu}_{m, m_l}^{sd}(t)}, 
\end{eqnarray*}
where $\tilde{\mu}_{m, m_u}^{sd}(t)$ and $\tilde{\mu}_{m, m_l}^{sd}(t)$ denote the actual numbers of flow $(s, d)$ packets sent from node $m$ to node $m_u$ and node $m_l$ at time $t$, respectively. 
\begin{lemma}\label{lemma:balanced-load}
 If the fictitious network is stable, then, 
\begin{itemize}
\item[(a)] For every node $m\in\cup_{j=1}^{n-1}\script{C}_j$,  
\begin{eqnarray}
\overline{\mu}_{m, m_u}^{\textsf{UU}} = \overline{\mu}_{m, m_l}^{\textsf{UL}},
\overline{\mu}_{m, m_u}^{\textsf{LU}} = \overline{\mu}_{m, m_l}^{\textsf{LL}}. \label{eq:per-node-rate}
\end{eqnarray}
\item[(b)] The average rate of any $(s, d)$ flow packets going through any partition node $m\in\script{C}_n$ satisfies $\overline{\mu}^{sd}_{m}=r_{sd}/2^{n-1}$. $\Diamond$
\end{itemize}
\end{lemma}
\begin{proof}
See Appendix A. 
\end{proof}

\vspace{-.1in}
\subsection{The arrival admission queue}
Since the arrivals to the network are dynamic, in order to perform packet admission in a fair manner, we  introduce an auxiliary variables $\gamma_{sd}(t)$ and create the following virtual \emph{admission queue} for every flow $(s, d)$: 
\begin{eqnarray}
H_{sd}(t+1) = \big[H_{sd}(t) -  R_{sd}(t) \big]^+ +\gamma_{sd}(t). 
\end{eqnarray}

Intuitively, $\gamma_{sd}(t)$ indicates how many flow $(s, d)$ packets should have been admitted into the network. However, due to the randomness of the arrivals, this may not be feasible at every time $t$. Hence, the admission queue $H_{sd}(t)$ is created to ensure that in the long run, the admitted packets have a rate that is no smaller than the rate they should have got.  

\vspace{-.1in}
\subsection{The output regulation queue}
Here we specify the last component needed for our algorithm. Note that the above subsections have been dealing with reducing the number of queues per node and balancing the traffic inside the Benes network. In order to guarantee stability of the network, one also needs to ensure that the total traffic going to any output port of the Benes network does not exceed its capacity. To do so, we create the following \emph{regulation queue} for each output port $d\in\{1, ..., 2^n\}$ (or equivalently,  output server $d$): 
\begin{eqnarray}
q_d(t+1) = \big[  q_d(t)   - (1-\eta) \big]^++ \sum_{s}R_{sd}(t).   
\end{eqnarray}
That is, the input to this queue are all the admitted packets destined for output port $d$, and the service rate of the queue is $1-\eta$ for some small $\eta>0$ for all time. 
The intuition here is that if these virtual queues are  stable, then the average traffic rate for any output port is no more than $1-\eta$. The reason we have a small $\eta$ ``slack'' is to ensure queue stability for the nodes in columns $n$ to $2n-1$ in the physical network. 

\vspace{-.1in}
\section{The Grouped-Backpressure algorithm (\tsf{G-BP})}\label{section:gbp}
In this section, we present the construction of the \tsf{G-BP} algorithm and its performance. 

\vspace{-.1in}
\subsection{Constructing \tsf{G-BP}} 
For notation purposes, we first define the aggregate network queue vector of the fictitious network as follows: 
\begin{eqnarray*}
\hspace{-.3in}&&\bv{Z}(t)=\big(Q_s^{\script{T}}(t), \,\forall\,s, \,\script{T}\in\Omega_s,\, Q_{m}^{\script{T}}(t),  \forall\,m\in\cup_{j=1}^{n-1}\script{C}_j,\,\script{T}\in\Omega_B, \\
\hspace{-.3in}&&\qquad Q_{m}^{D_1}(t), Q_{m}^{D_2}(t), \,\forall\, m\in\script{C}_n,\, H_{sd}(t), \,\forall\,(s, d),  \, q_d(t), \forall\,d  \big). \nonumber
\end{eqnarray*}
%
Then,  we define the following Lyapunov function: 
\begin{eqnarray}
\hspace{-.45in}&&L(t) \triangleq \frac{1}{2} \sum_{s, \script{T}\in\Omega_s}[Q_{s}^{\script{T}}(t)]^2+ \frac{1}{2} \sum_{m\in\cup_{j=1}^{n-1}\script{C}_j}\sum_{\script{T}\in\Omega_B}[Q_{m}^{\script{T}}(t)]^2 \\
\hspace{-.45in}&&\quad +\frac{1}{2}\sum_{m\in\script{C}_n}\sum_{i=1,2}[Q_{m}^{D_i}(t)]^2+ \frac{1}{2}\sum_{s,d}[H_{sd}(t)]^2 + \frac{1}{2}\sum_d[q_d(t)]^2. \nonumber
\end{eqnarray}
Now define a Lyapunov drift as follows: 
\begin{eqnarray}
\Delta(t)  \triangleq \expect{L(t+1) - L(t)\left.|\right.\bv{Z}(t)}. 
\end{eqnarray}
Using the facts that $0\leq A_{sd}(t)\leq A_{\max}$ and that all the link capacities in the network are bounded, we get the following lemma for the drift. In the lemma, the parameter $V\geq1$ is a control parameter offered by the algorithm to control the flow utility performance.
\begin{lemma}\label{lemma:drift-ineq}
Under any control policy, the following property holds for the drift at any time $t$: 
\begin{eqnarray}
\hspace{-.3in}&&\Delta(t) - V\expect{\sum_{s, d} U_{sd}(\gamma_{sd}(t))\left.|\right.\bv{Z}(t) }\label{eq:drift-utility}\\
\hspace{-.3in}&&\leq B  - \sum_{d}q_d(t)(1-\eta)-\sum_{m\in\script{C}_n, i}Q_{m}^{D_i}(t)\expect{   \mu_{m, D_i}(t)   \left.|\right.\bv{Z}(t)}\nonumber\\
\hspace{-.3in}&&\qquad -\sum_{s, d} \expect{VU_{sd}(\gamma_{sd}(t)) - H_{sd}(t) \gamma_{sd}(t)\left.|\right.\bv{Z}(t) }\nonumber\\
\hspace{-.3in}&&\qquad - \sum_s\sum_{d\leq2^{n-1}}\expect{ R_{sd}(t)\big[ H_{sd}(t) -  q_d(t) -Q_{s}^{\tsf{U}}(t)  \big]     \left.|\right.\bv{Z}(t) }\nonumber\\
\hspace{-.3in}&&\qquad - \sum_s\sum_{d>2^{n-1}}\expect{ R_{sd}(t)\big[ H_{sd}(t) -  q_d(t) -Q_{s}^{\tsf{L}}(t)  \big]     \left.|\right.\bv{Z}(t) }\nonumber\\
\hspace{-.3in}&& - \sum_s\expect{\mu_{s, m(s)}^{\tsf{U}}(t)  \big[ Q_s^{\tsf{U}}(t)  -\frac{1}{2}Q^{\textsf{UU}}_{m(s)}(t) - \frac{1}{2}Q^{\textsf{UL}}_{m(s)}(t) \big] \left.|\right.\bv{Z}(t)  } \nonumber\\
\hspace{-.3in}&& - \sum_s\expect{\mu_{s, m(s)}^{\tsf{L}}(t)  \big[ Q_s^{\tsf{L}}(t)  -\frac{1}{2}Q^{\textsf{LU}}_{m(s)}(t) - \frac{1}{2}Q^{\textsf{LL}}_{m(s)}(t) \big] \left.|\right.\bv{Z}(t)  } \nonumber\\
\hspace{-.3in}&&- \sum_{m\in\script{C}_{n-1}}\sum_{\script{T}\in\{\tsf{UU, UL}\}}\expect{  \mu^{\script{T}}_{m, m(\script{T})}(t)  \big[    Q_m^{\script{T}}(t) -  Q^{D_1}_{m(\script{T})}(t) 
\big]    \left.|\right.\bv{Z}(t)  }\nonumber\\
\hspace{-.3in}&& - \sum_{m\in\script{C}_{n-1}}\sum_{\script{T}\in\{\tsf{LU, LL}\}}\expect{  \mu^{\script{T}}_{m, m(\script{T})}(t)  \big[    Q_m^{\script{T}}(t) - Q^{D_2}_{m(\script{T})}(t) \big]    \left.|\right.\bv{Z}(t)  }\nonumber\\
\hspace{-.3in}&& - \sum_{m\in\cup_{j=1}^{n-2}\script{C}_{j}}\sum_{\script{T}\in\{\tsf{UU, UL}\}}\expect{  \mu^{\script{T}}_{m, m(\script{T})}(t)  \big[    Q_m^{\script{T}}(t) -  \frac{1}{2}Q^{\textsf{UU}}_{m(\script{T})}(t) \nonumber\\
\hspace{-.3in}&&\qquad\qquad\qquad\qquad\qquad\qquad\qquad\qquad - \frac{1}{2}Q^{\textsf{UL}}_{m(\script{T})}(t) 
\big]    \left.|\right.\bv{Z}(t)  }\nonumber\\
\hspace{-.3in}&& - \sum_{m\in\cup_{j=1}^{n-2}\script{C}_{j}}\sum_{\script{T}\in\{\tsf{LU, LL}\}}\expect{  \mu^{\script{T}}_{m, m(\script{T})}(t)  \big[    Q_m^{\script{T}}(t) -  \frac{1}{2}Q^{\textsf{LU}}_{m(\script{T})}(t) \nonumber\\
\hspace{-.3in}&&\qquad\qquad\qquad\qquad\qquad\qquad\qquad\qquad - \frac{1}{2}Q^{\textsf{LL}}_{m(\script{T})}(t) 
\big]    \left.|\right.\bv{Z}(t)  }.\nonumber
\end{eqnarray}
Here $B$ is a constant given by: 
\begin{eqnarray}
B=\frac{1}{2}[2^{n}(10n-2) +A_{\max}^2(2^{3n-1}+2^{2n+1}+2^{3n})],  \label{eq:B-def}
\end{eqnarray}
and the expectation is taken over the random arrivals as well as the potential randomness in the actions. $\Diamond$
\end{lemma}
\begin{proof}
See Appendix B. 
\end{proof}

Note that since the Benes network size is $\Theta(2^n)$, the $n$ value is only logarithmic in the network size. Hence, $B$ is indeed only polynomial in the network size. 
Based on the above lemma, we now describe our algorithm for the \emph{physical} system. In the algorithm, we will operate the nodes in $\script{S}$ and $\cup_{j=1}^{n-1}\script{C}_j$ in the physical system exactly as we operate them in the fictitious system. 
For these nodes, the actions will be chosen in every time slot to minimize the right-hand-side (RHS) of the drift expression (\ref{eq:drift-utility}). 
For all the modules in columns $n$ to $2n-1$, we simply do a free-flow routing.  
%
%

\underline{\textsf{Grouped-Backpressure (G-BP)}}
At every time slot $t$, observe $\bv{A}(t)$ and $\bv{Z}(t)$, and perform the following: 
\begin{itemize}
\item \underline{\textsf{Auxiliary  Variable Selection:}} For every $(s, d)$ flow, choose  $\gamma_{sd}(t)$ to solve: 
\begin{eqnarray}
\max: &&   VU_{sd}(\gamma_{sd}(t)) - H_{sd}(t)\gamma_{sd}(t) \label{eq:gbp-gamma}\\
\text{s.t.}  && 0\leq \gamma_{sd}(t)\leq A_{\max}. \nonumber
\end{eqnarray}

\item \underline{\textsf{Admission Control:}} For every input server $s$: If $d\leq2^{n-1}$, choose $R_{sd}(t)=A_{sd}(t)$ if $H_{sd}(t)    -  q_d(t) -  Q_{s}^{\textsf{U}}(t)  >0$; else  choose $R_{sd}(t)=0$. If $d>2^{n-1}$, choose $R_{sd}(t)=A_{sd}(t)$ if $H_{sd}(t)    -  q_d(t) -  Q_{s}^{\textsf{L}}(t)  >0$; else  choose $R_{sd}(t)=0$.

\item \underline{\textsf{Routing and Scheduling:}} 
\begin{itemize}
\item \underline{For any node $m\in\cup_{j=1}^{n-1}\script{C}_j\cup\script{S}$:}  define the following weights for the outgoing link $[m, m_u]$: 
\begin{eqnarray}
\hspace{-.3in}&&W^{\tsf{U}}_{m, m_u}(t)\triangleq \max\bigg[   Q^{\textsf{UU}}_{m}(t) - \tilde{Q}_{m_u}^{\tsf{U}}(t), 0 \bigg], \label{eq:weight-1-up}\\
\hspace{-.3in}&&W^{\tsf{L}}_{m, m_u}(t)\triangleq \max\bigg[   Q^{\textsf{LU}}_{m}(t) -\tilde{Q}_{m_u}^{\tsf{L}}(t), 0 \bigg],  \label{eq:weight-2-up}
\end{eqnarray}
where $\tilde{Q}_{m_u}^{\tsf{U}}(t)$ and $\tilde{Q}_{m_u}^{\tsf{L}}(t)$ are defined as: 
\begin{eqnarray}
\hspace{-.3in}\tilde{Q}_{m_u}^{\tsf{U}}(t) = \left\{\begin{array}{cc}
\frac{1}{2}Q^{\textsf{UU}}_{m_u}(t) + \frac{1}{2}Q^{\textsf{UL}}_{m_u}(t) & j_m\leq n-2,\\
Q^{D_1}_{m_u}(t) & j_m=n-1, 
\end{array}\right.\label{eq:sum-queue-1}\\
\hspace{-.3in}\tilde{Q}_{m_u}^{\tsf{L}}(t) = \left\{\begin{array}{cc}
\frac{1}{2}Q^{\textsf{LU}}_{m_u}(t) + \frac{1}{2}Q^{\textsf{LL}}_{m_u}(t) & j_m\leq n-2,\\
Q^{D_2}_{m_u}(t) & j_m=n-1. 
\end{array}\right.\label{eq:sum-queue-2}
\end{eqnarray}
Then, we choose the service rates $\mu_{m, m_{u}}^{\textsf{UU}}(t)$ and $\mu_{m, m_{u}}^{\textsf{LU}}(t)$ for link $[m, m_u]$ to solve: 
\begin{eqnarray}
\hspace{-.45in}&&\max: \quad\mu_{m, m_{u}}^{\textsf{UU}}(t)W_{m, m_{u}}^{\tsf{U}} + \mu_{m, m_{u}}^{\textsf{LU}}(t)W_{m, m_{u}}^{\tsf{L}} \label{eq:gbp-routing}\\
\hspace{-.45in}&&\quad\text{s.t.} \quad\,\,\,\mu_{m, m_{u}}^{\textsf{UU}} + \mu_{m, m_{u}}^{\textsf{LU}}\leq 1,\,\, \mu_{m, m_{u}}^{\textsf{UU}}, \mu_{m, m_{u}}^{\textsf{LU}}\in\{0, 1\}. \nonumber
\end{eqnarray}
To solve for $\mu_{m, m_{l}}^{\textsf{UL}}(t)$ and $\mu_{m, m_{l}}^{\textsf{LL}}(t)$, we replace $Q_m^{\tsf{UU}}(t)$ and $Q_m^{\tsf{LU}}(t)$ with $Q_m^{\tsf{UL}}(t)$  and $Q_m^{\tsf{LL}}(t)$ in (\ref{eq:weight-1-up}) and  (\ref{eq:weight-2-up}). Also, we replace $m_u$ and $D_1$ with $m_l$ and $D_2$ in (\ref{eq:sum-queue-1}) and (\ref{eq:sum-queue-2}). If  $m=s\in\script{S}$, we simply replace $Q_m^{\tsf{UU}}(t)$ and $Q_m^{\tsf{UL}}(t)$ with $Q_s^{\tsf{U}}(t)$  and $Q_s^{\tsf{L}}(t)$ in (\ref{eq:weight-1-up}) and  (\ref{eq:weight-2-up}), and replace $m_u$ by $m(s)$ in (\ref{eq:sum-queue-1}) and (\ref{eq:sum-queue-2}). 


\item \underline{For every node  $m\in\cup_{j=n}^{2n-1}\script{C}_j$:}  Each module serves each FIFO queue for each outgoing link according to (\ref{eq:queue-dyn-phy-1}) and (\ref{eq:queue-dyn-phy-2}) with $\mu_{m, m_u}(t)=\mu_{m, m_l}(t)=1$ for all time. 

\end{itemize}

\item \underline{\tsf{Queue Updates:}} In the fictitious system, choose the service rates $\mu_{m, D_1}(t)$ and $\mu_{m, D_2}(t)$ to solve: 
\begin{eqnarray}
\max: && Q_{m}^{D_1}(t)  \mu_{m, D_1}(t) + Q_{m}^{D_2}(t)  \mu_{m, D_2}(t) \label{eq:gbp-routing-partition}\\
\text{s.t.} &&  \mu_{m, D_1}(t),  \mu_{m, D_2}(t)\in\{0, 1\}. \nonumber
\end{eqnarray}
Then, update all the queues in both the fictitious system and the physical system according to their dynamics. $\Diamond$
\end{itemize}
We note that \textsf{G-BP} only controls the first half of the  physical system with the backpressure actions. 
All the nodes in columns $n$ to $2n-1$ simply serve the flows with a ``free-flow'' manner, i.e., always serve the flows at the maximum rate. 
This is different from the usual backpressure algorithms that control all the queues  in the network to ensure stability. 
%

\vspace{-.1in}
\subsection{Performance analysis}\label{section:analysis}
In this section, we prove that \textsf{G-BP} achieves a near-optimal performance. To carry out our analysis, we first have the following theorem, which characterizes the capacity region of a Benes network. In the theorem, we use $\bv{r}=(r_{sd}, \,\forall\, (s,d))$ to denote the vector of arrival rates, where $r_{sd}$ represents the average rate of the $(s, d)$ flow. 
\begin{theorem} \cite{benes-net} \cite{benes-net-book}
The capacity region of the Benes network $\mathbb{B}_n$ is given by:
\begin{eqnarray*}
\Lambda_n=\{\bv{r}\left.|\right. \sum_{s=1}^{2^n}r_{sd}\leq1, \,\sum_{d=1}^{2^n}r_{sd}\leq1, \, r_{sd}\geq0, \forall\,s,d\}.  \Diamond
\end{eqnarray*}
\end{theorem}

We now present the performance results of the \textsf{G-BP} algorithm. 
Recall that $\beta$ is defined in (\ref{eq:beta-def}) to be the maximum first derivative among all utility functions, and that $\bv{r}^{\textsf{opt}}\in\Lambda_n$ denotes the optimal solution to the flow utility maximization problem.  
\begin{theorem}\label{theorem:gbp-per}
Suppose both the fictitious network and the physical network are empty at time $t=0$, i.e., 
all the queues are zero. Then, 
(i) Both the fictitious network and the physical network are stable under \textsf{G-BP}, and 
(ii) Denote $\bv{r}^{\textsf{G-BP}}$ the time average rate vector achieved by \textsf{G-BP}. We have: 
\begin{eqnarray}
U(\bv{r}^{\textsf{G-BP}})\geq U(\bv{r}^{\textsf{opt}})-\frac{B}{V}-2^{n}\beta\eta. \quad\Diamond\label{eq:utility-gbp}
\end{eqnarray}
\end{theorem}
\begin{proof}
See Appendix C. 
\end{proof}
From (\ref{eq:utility-gbp}), we see that the utility performance of \tsf{G-BP} can arbitrarily approach the optimal as we increase $V$ and decrease $\eta$. However, doing so will increase the average network delay. Hence, there is a natural tradeoff between the utility performance and the network delay. 

Note that though the performance results in Theorem \ref{theorem:gbp-per} look similar to previous results in \cite{neelyfairness}, the proof is indeed quite different. This is because in our case, we impose a special queueing structure on the network, and the second half of the network uses a free-flow routing. These two features make the analysis very different from the usual backpressure algorithms, 
under which each node maintains a separate queue for each flow, and all the network actions are based on the network queue sizes. 

\vspace{-.1in}
\subsection{Discussion on implementation}
We note that the \tsf{G-BP} algorithm can easily be implemented in a fully distributed manner. 
Specifically, one can maintain the virtual admission queues at the input servers and maintain the virtual output regulation queues at the output servers using counters, as shown in Fig. \ref{fig:implement}. 
%
With this arrangement, the auxiliary variable selection step can easily be done locally at the input servers, and the routing and scheduling step can easily be done by each node exchanging queue information only with its four neighbors. 
The admission control step requires the input servers to know the regulation queue sizes. This can be achieved by message passing the regulating queue sizes along the network using prioritized packets. Similarly, the update of the regulation queues requires the knowledge of the arrivals for the output port. 
This can be approximated by using the arrivals to the output servers as the input to the regulation queues. Though message passing and queue approximation may incur performance loss in practice, we will see in the simulation section that, the \tsf{G-BP} algorithm is indeed very robust and can still achieve near-optimal performance even under different message passing delays and regulation queue approximation. 

Finally, note that though we have described implementing our algorithm with actual data queue sizes. In practice, to further reduce network delay, we can also implement \tsf{G-BP} with counters to keep track of the queue processes that should have been generated for decision making, and admit slightly smaller arrival rates than \tsf{G-BP}. 
\begin{figure}[cht]
\centering
\vspace{-.1in}
\includegraphics[height=2.0in, width=3.4in]{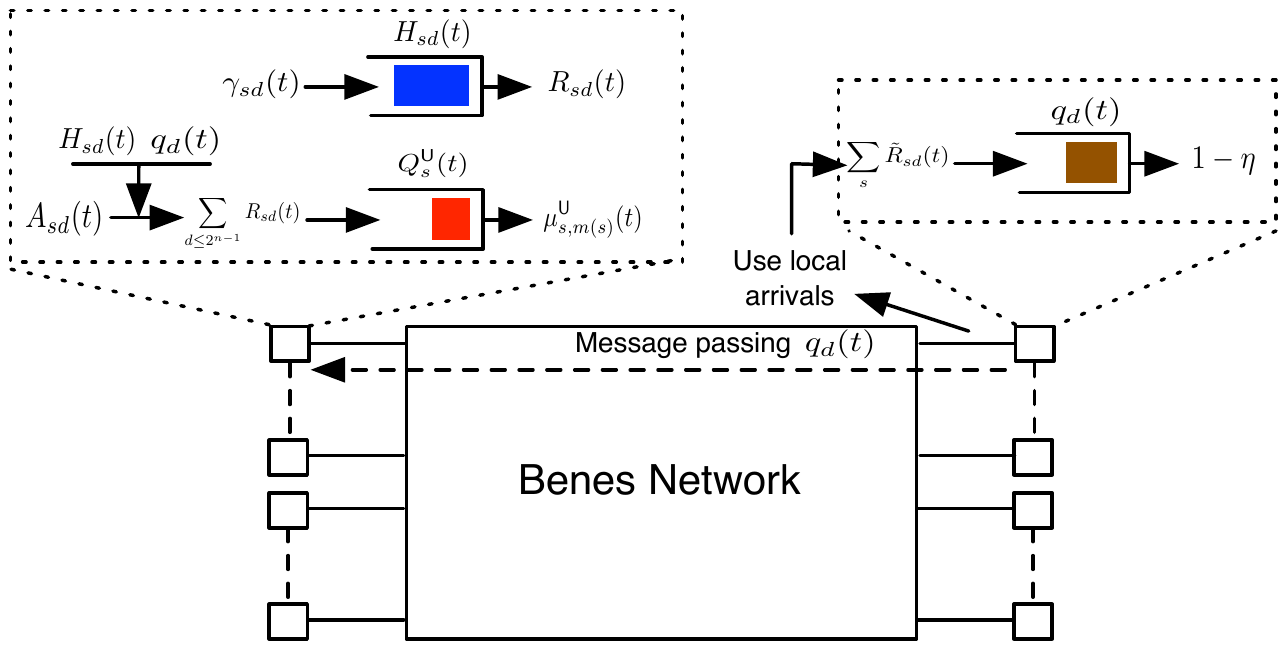}
\vspace{-.1in}
\caption{Implementation of \tsf{G-BP}. The virtual admission queues are maintained at the input servers, while the virtual regulation queues are maintained at the output servers. Message passing is used to send regulation queue information through the network for admission control. The regulation queues can use the local arrivals to the output servers as the input.}
\label{fig:implement}
\vspace{-.1in}
\end{figure}

\vspace{-.1in}
\section{Simulation}\label{section:simulation}
In this section, we present the simulation results of \tsf{G-BP}  on a $2^4\times2^4$ size Benes network. For simplicity, we assume that $A_{sd}(t)=A_{\max}=2$ for all time. 

In the simulation, we assume that every flow has a utility function $\log(1+r_{sd})$. In every time slot, each flow can admit $0$, $1$ or $2$ packets.  We simulate the system for $V\in\{5, 10, 20, 50, 100 \}$ and $\eta=0.01$. Each simulation is run for $10^5$ slots. 
%
To test the robustness of \tsf{G-BP} against the delay and sparsity in message passing and the regulation queue approximation, we simulate four different cases. 
(i) The original \tsf{G-BP} algorithm, where the message passing delay is zero and the regulation queue is exact. (ii) The case when the input to the regulation queue $q_d(t)$ are the actual packet arrivals to the output server $d$ (the service rate is still $1-\eta$), and admission control at time $t$ uses $q_d(t-(2n-1))$ instead of $q_d(t)$. (iii) Similar to the second case, but  admission control at time $t$ uses $q_d(t-5(2n-1))$. (iv) Similar to the second case, but the regulation queue information is only sent every $5(2n-1)$ slots and has a delay of $5(2n-1)$. That is, admission control at time $t$ uses $q_d(t_0)$ where $t_0=\max[(\lfloor \frac{t}{5(2n-1)}\rfloor-1)5(2n-1), 1]$.

\begin{figure}[cht]
\centering
\vspace{-.1in}
\includegraphics[height=2in, width=3.4in]{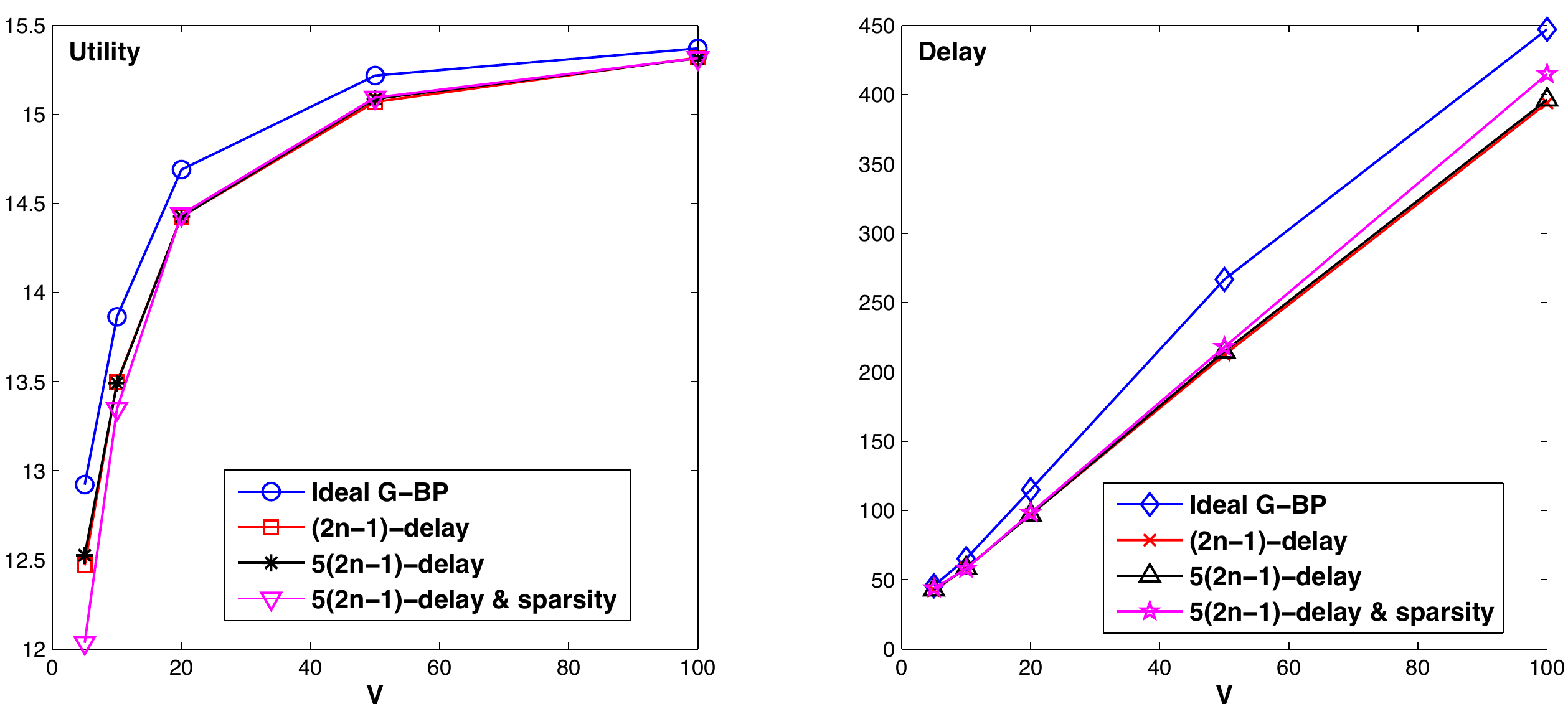}
\vspace{-.1in}
\caption{The  aggregate flow utility  and average packet delay under \tsf{G-BP}. One can see that \tsf{G-BP} works very well even with message passing delay and regulation queue approximation.}
\label{fig:delay-utility-gbp}
\vspace{-.1in}
\end{figure}
Fig. \ref{fig:delay-utility-gbp} shows the performance of the \tsf{G-BP} algorithm. Here the average delay (in number of slots) is computed using the set of packets that are delivered when the simulation ends. For all simulations, this set contains more than $99.9\%$ of the total packets that enter the network. We see that 
as we increase the $V$ value, the aggregate flow utility quickly converges to its optimal value. However, doing so also leads to a linear increase of the average packet delay. We also see from the figure that, \tsf{G-BP} is indeed very robust to the delay and sparsity in  message passing, and regulation queue approximation. 

In Fig. \ref{fig:queue-process}, we plot a recorded queue process of the network under \tsf{G-BP}  for $V=10$. In this case, we change each flow's utility function to $w_{sd}\log(1+r_{sd})$ in the middle of the simulation, where $w_{sd}$ takes values $1, 2$ or $3$ equally likely.  
We  see that after the change, \tsf{G-BP} quickly adapts to the new utility functions  and performs admission and routing accordingly. 
\begin{figure}[cht]
  \centering
  \vspace{-.1in}
\includegraphics[height=2in, width=3.4in]{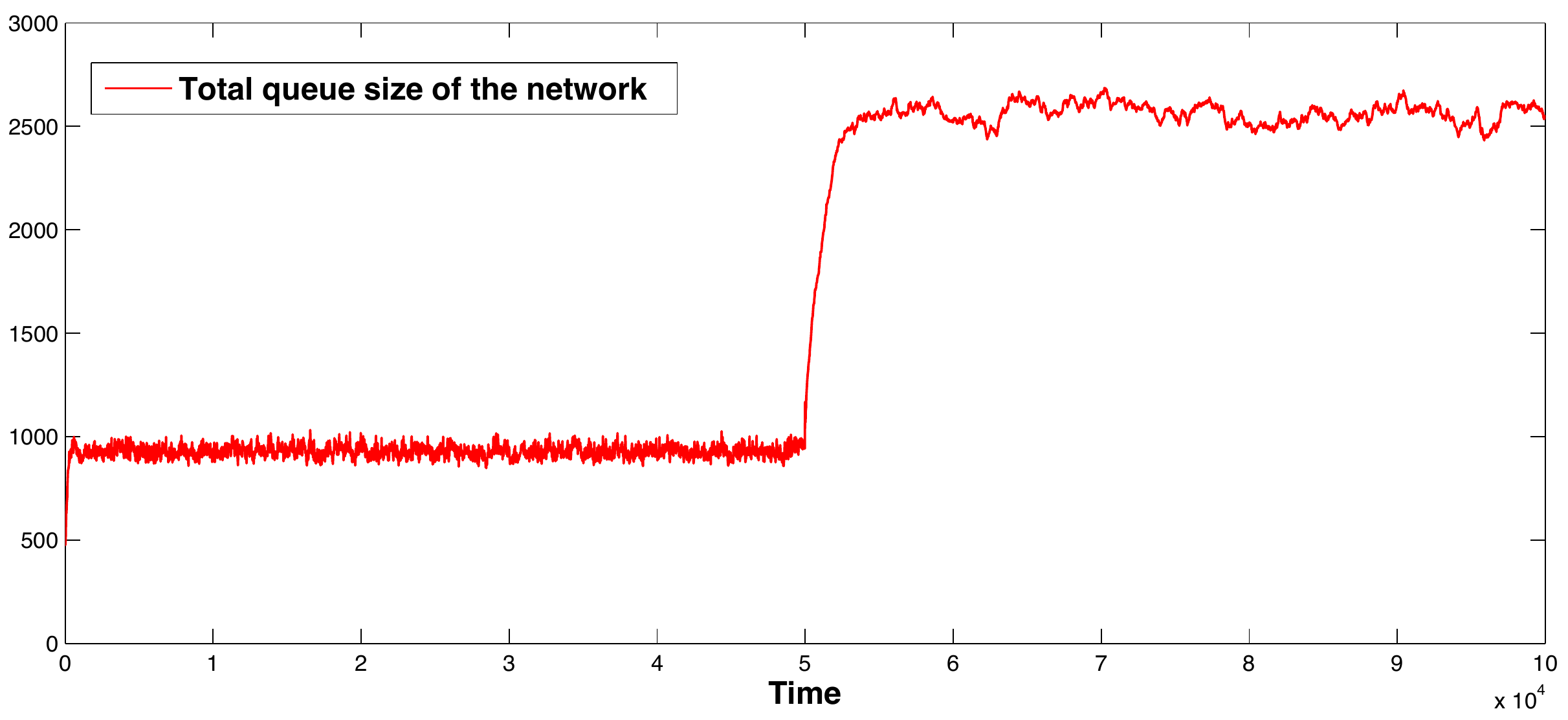}
\vspace{-.15in}
  \caption{The total network queue size under \tsf{G-BP} with $V=10$.}
  \label{fig:queue-process}
  \vspace{-.1in}
\end{figure}

Finally, we also evaluate the average packet delay as a function of the network size, to see how the algorithm scales. As comparison, we also simulate an ``enhanced'' \tsf{G-BP} algorithm, in which we replace each queue value in the algorithm with the queue value plus the node's hop count to the destination, i.e., its column number. The idea is to create ``bias'' towards the packet destinations. This enhancement is similar to the EDRPC algorithm developed in \cite{neelypowerjsac}. 
We can see from Fig. \ref{fig:scaling-bias} that the average packet delay under \tsf{G-BP} scales as $\Theta(n^2)$. Since the Benes network size is $\Theta(2^n)$, this implies that the \emph{average delay grows only logarithmically in the network size}. 
\begin{figure}[cht]
\centering
\vspace{-.1in}
\includegraphics[height=2.0in, width=3.4in]{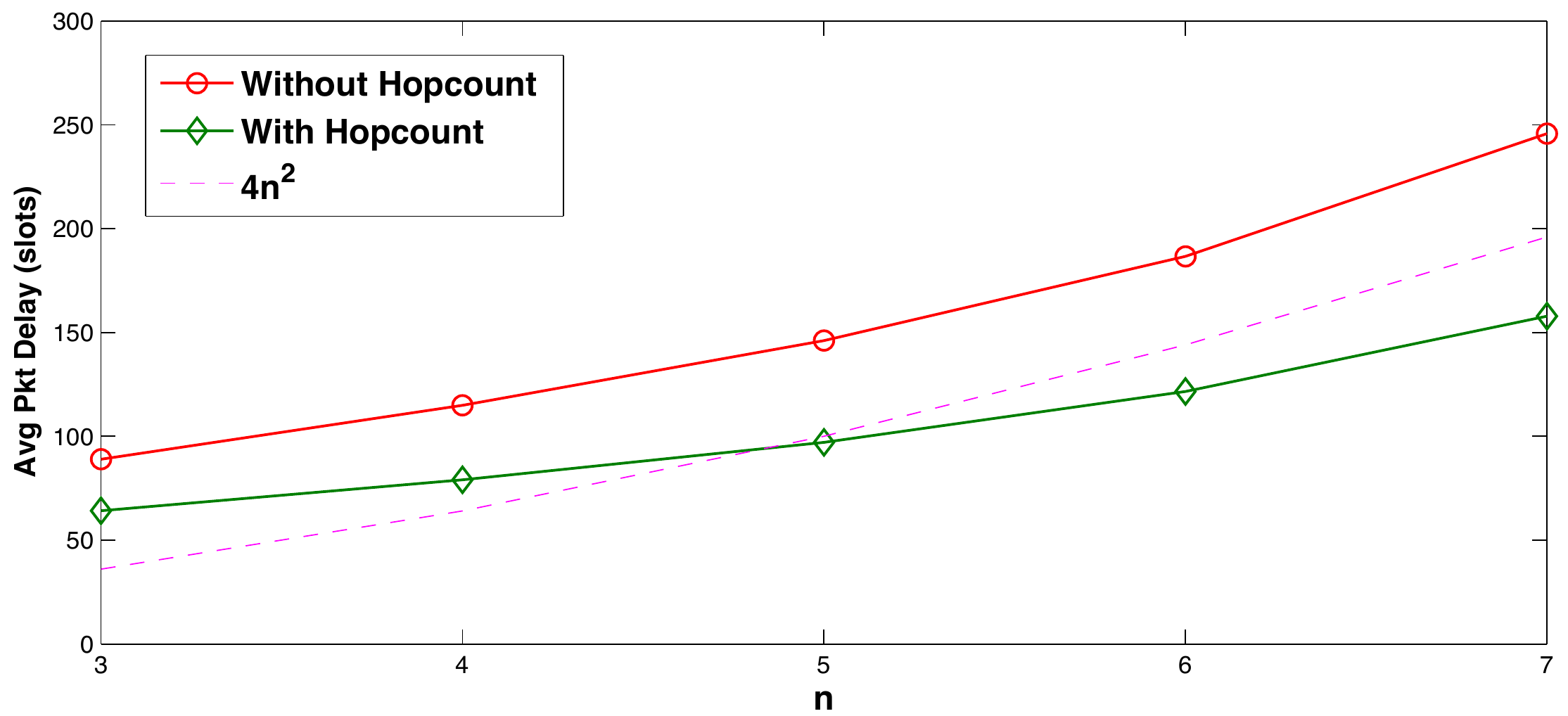}
\vspace{-.1in}
\caption{Average delay as a function of the network size under $V=10$.}
\label{fig:scaling-bias}
\vspace{-.1in}
\end{figure}

Note that in Fig. \ref{fig:scaling-bias} we have plotted the average delay in number of slots. To get some physical understanding of the results, assume that each packet has $500$ bytes and each link has a capacity of $1$ Gbit/second, which are both quite common in practice. Then, every slot is $4$ microseconds. Hence, we see that the average packet delay under Benes network with  \tsf{G-BP} is roughly $1$ millisecond when the network size is $128\times128$. 
This demonstrates the good delay performance of our network design approach. 

\vspace{-.1in}
\section{Conclusion}\label{section:conclusion}
In this work, we develop a novel networking solution called \emph{Benes packet network}, which consists of a Benes network built with simple commodity switches, a flow utility maximization mechanism, and a Grouped-Backpressure (\tsf{G-BP}) routing and scheduling algorithm. We show that this combination can achieve a near-optimal flow utility and ensure small end-to-end delay for the traffic flows. Our approach also only requires each switch module to maintain at most four queues regardless of the network size, and can easily be implemented in practice in a fully distributed manner. 

\vspace{-.1in}
\section*{Appendix A -- Proof of Lemma \ref{lemma:balanced-load}}
\begin{proof} (Lemma \ref{lemma:balanced-load}) 
We first prove Part (a). From the queueing dynamic equation  (\ref{eq:queue-bernes-net}), we see that for any node $m\in\cup_{j=1}^{n-1}\script{C}_j$, the input rates into $Q_{m}^{\textsf{UU}}(t)$ and $Q_{m}^{\textsf{UL}}(t)$ are equal because of  random splitting.  
Similarly, the input rates into $Q_{m}^{\textsf{LU}}(t)$ and $Q_{m}^{\textsf{LL}}(t)$ are the same. Hence, if the fictitious network is stable,  the output rates from these queues are equal to their input rates \cite{neelynowbook}. Therefore (\ref{eq:per-node-rate}) holds. 

Now we prove Part (b) by induction. 
 First we see that it holds for any $4\times4$ Benes network. This is because if the fictitious network is stable, then the input switch modules split the incoming flows equally into the two partition nodes (see Fig. \ref{fig:benes-demo}). 

Now suppose the same is true for a $2^{n-1}\times 2^{n-1}$ Benes network, we want to show that it also holds for a $2^n\times2^n$ Benes network. 

To see this, note from Fig. \ref{fig:benes-demo} that each $2^n\times2^n$ Benes network consists of two $2^{n-1}\times2^{n-1}$ subnetworks, $2^{n-1}$ input switch modules and $2^{n-1}$  output switch modules. 
According to the structure of the Benes network, any input switch module has one link connecting to the upper $2^{n-1}\times2^{n-1}$ subnetwork and the other one connecting to the lower $2^{n-1}\times2^{n-1}$ subnetwork. From Part (a), we see that half of a flow's rate will be routed through the upper subnetwork and the other half will be routed through the lower subnetwork. 
Now consider  the upper subnetwork and view the flow traffic into this subnetwork as its own external input. Since this subnetwork is also stable, the flow's traffic will be equally split and routed via its partition nodes by induction. Since all the partition nodes coincide according to Fact \ref{fact:partition-node-coincide}, we see that the lemma follows. 
\end{proof}

\vspace{-.1in}
\section*{Appendix B -- Proof of Lemma \ref{lemma:drift-ineq}}
Here we present the proof of Lemma \ref{lemma:drift-ineq}. 
\begin{proof}
Squaring both sides of (\ref{eq:queue-source-q}) and using the fact that for any real number $x$, $([x]^+)^2\leq x^2$, we get for every $s\in\script{S}$ and $\script{T}\in\Omega_s$ that: 
\begin{eqnarray}
[Q_{s}^{\script{T}}(t+1)]^2\leq [Q_{s}^{\script{T}}(t)]^2 + [R_{s}^{\script{T}}(t)]^2 + [\mu_{s, m(s)}^{\script{T}}(t)]^2 \label{eq:square-q}\\
- 2Q_{s}^{\script{T}}(t)[\mu_{s, m(s)}^{\script{T}}(t) - R_{s}^{\script{T}}(t)]. \nonumber
\end{eqnarray}
Now note that $\mu_{s, m(s)}^{\script{T}}(t)\leq1$ and $R_{s}^{\script{T}}(t)\leq 2^{n-1}A_{\max}$. Hence, if we define $B_1\triangleq 2(2^n+2^{3n-2}A^2_{\max})$ and sum (\ref{eq:square-q}) over $s\in\script{S}$ and $\script{T}\in\Omega_s$, we have: 
\begin{eqnarray*}
\hspace{-.4in}&&\sum_{s\in\script{S}, \script{T}\in\Omega_s}[Q_{s}^{\script{T}}(t+1)]^2\leq \sum_{s\in\script{S},\script{T}\in\Omega_s}[Q_{s}^{\script{T}}(t)]^2 + B_1\\
\hspace{-.4in}&&\qquad\qquad\qquad - 2\sum_{s\in\script{S},\script{T}\in\Omega_s}Q_{s}^{\script{T}}(t)[\mu_{s, m(s)}^{\script{T}}(t) - R_{s}^{\script{T}}(t)]. \label{eq:square-q2}
\end{eqnarray*}
Using a similar argument as above, we get the following: 
\begin{eqnarray*}
\hspace{-.4in}&&\sum_{m\in\cup_{j=1}^{n-1}\script{C}_j, \script{T}\in\Omega_B}[Q_{m}^{\script{T}}(t+1)]^2-\sum_{m\in\cup_{j=1}^{n-1}\script{C}_j, \script{T}\in\Omega_B}[Q_{m}^{\script{T}}(t)]^2 \\
\hspace{-.4in}&& \leq B_2 - 2\sum_{m\in\cup_{j=1}^{n-1}\script{C}_j, \script{T}\in\Omega_B}Q_{m}^{\script{T}}(t)[\mu_{m, m(\script{T})}^{\script{T}}(t) - R_{m}^{\script{T}}(t)]. \label{eq:square-q2}
\end{eqnarray*}
Here $B_2\triangleq 20(n-1)2^{n-1}$. 
Now repeat the above for all the other queues in the fictitious system, we will also get: 
\begin{eqnarray*}
\hspace{-.4in}&&\sum_{m\in\script{C}_n,i}[Q_{m}^{D_i}(t+1)]^2 - \sum_{m\in\script{C}_n,i}[Q_{m}^{D_i}(t)]^2\\
\hspace{-.4in}&&\qquad\qquad\qquad\leq B_3-2\sum_{m\in\script{C}_n,i}Q_m^{D_i}(t)[ \mu_{m, D_i}(t) - R^{D_i}_m(t)], \\
\hspace{-.4in}&&\sum_{s, d}[H_{sd}(t+1)]^2 - \sum_{s,d}[H_{sd}(t)]^2\\
\hspace{-.4in}&&\qquad\qquad\qquad\leq B_4-2\sum_{s,d}H_{sd}(t)[ R_{sd}(t) - \gamma_{sd}(t)], \\
\hspace{-.4in}&&\sum_{d}[q_{d}(t+1)]^2 - \sum_{d}[q_{d}(t)]^2\\
\hspace{-.4in}&&\qquad\qquad\qquad\leq B_5-2\sum_{d}q_{d}(t)[ 1-\eta - \sum_{s}R_{sd}(t)]. 
\end{eqnarray*}
Here $B_3=5\cdot2^n$, $B_4=2^{2n+1}A_{\max}^2$ and $B_5=2^n+2^{3n}A_{\max}^2$. 

Summing all the resulting inequalities, multiplying both sides by $\frac{1}{2}$ and taking expectations on both sides conditioning on $\bv{Z}(t)$, we obtain the following: 
\begin{eqnarray}
\hspace{-.2in}&&\Delta(t) \leq B  \label{eq:drift-ineq-step1}\\
\hspace{-.2in}&&- \sum_{s\in\script{S}}\sum_{\script{T}\in\Omega_s} Q_{s}^{\script{T}}(t)\expect{   \mu_{s, m(s)}^{\script{T}}(t)   - R_s^{\script{T}}(t)\left.|\right. \bv{Z}(t) } \nonumber \\
\hspace{-.2in}&&- \sum_{m\in\cup_{j=1}^{n-1}\script{C}_j}\sum_{\script{T}\in\Omega_B} Q_{m}^{\script{T}}(t)\expect{   \mu_{m, m(\script{T})}^{\script{T}}(t)   - R_m^{\script{T}}(t)\left.|\right. \bv{Z}(t) } \nonumber \\
\hspace{-.2in}&&\qquad -\sum_{m\in\script{C}_n} \sum_{i=1,2}Q_{m}^{D_i}(t)\expect{   \mu_{m, D_i}(t)   - R_{m}^{D_i}(t) \left.|\right. \bv{Z}(t) }  \nonumber\\
\hspace{-.2in}&&\qquad-\sum_{s, d} H_{sd}(t) \expect{  R_{sd}(t) -\gamma_{sd}(t) \left.|\right. \bv{Z}(t)}\nonumber \\
\hspace{-.2in}&&\qquad -\sum_{d} q_d(t)  \expect{  1-\eta - \sum_{s}R_{sd}(t) \left.|\right. \bv{Z}(t)}. \nonumber
\end{eqnarray}
Here the constant $B\triangleq\frac{1}{2}\sum_{i=1,...,5}B_i$, i.e.,  
\begin{eqnarray}
B=\frac{1}{2}[2^{n}(10n-2) +A_{\max}^2(2^{3n-1}+2^{2n+1}+2^{3n})]. 
\end{eqnarray}
Now by adding to both sides of (\ref{eq:drift-ineq-step1}) the term $-V\expectm{\sum_{sd}U_{sd}(\gamma_{sd}(t))\left.|\right.\bv{Z}(t)}$, we get: 
\begin{eqnarray}
\hspace{-.2in}&&\Delta(t) - V\expect{\sum_{s, d} U_{sd}(\gamma_{sd}(t))\left.|\right.\bv{Z}(t) }\label{eq:drift-utility-rate-form}\\
\hspace{-.2in}&&\leq B  - \sum_{s}\sum_{\script{T}\in\Omega_s} Q_{s}^{\script{T}}(t)\expect{   \mu_{s, m(s)}^{\script{T}}(t)   - R_s^{\script{T}}(t)\left.|\right. \bv{Z}(t) } \nonumber \\
\hspace{-.2in}&&- \sum_{m\in\cup_{j=1}^{n-1}\script{C}_j}\sum_{\script{T}\in\Omega_B} Q_{m}^{\script{T}}(t)\expect{   \mu_{m, m(\script{T})}^{\script{T}}(t)   - R_m^{\script{T}}(t)\left.|\right. \bv{Z}(t) } \nonumber \\
\hspace{-.2in}&&\qquad -\sum_{m\in\script{C}_n} \sum_{i=1,2}Q_{m}^{D_i}(t)\expect{   \mu_{m, D_i}(t)   - R_{m}^{D_i}(t) \left.|\right. \bv{Z}(t) }  \nonumber\\
\hspace{-.2in}&&\qquad - V\expect{\sum_{s, d} U_{sd}(\gamma_{sd}(t))\left.|\right.\bv{Z}(t) }\nonumber \\
\hspace{-.2in}&&\qquad-\sum_{s, d} H_{sd}(t) \expect{  R_{sd}(t) -\gamma_{sd}(t) \left.|\right. \bv{Z}(t)}\nonumber \\
\hspace{-.2in}&&\qquad -\sum_{d} q_d(t)  \expect{  1-\eta - \sum_{s}R_{sd}(t) \left.|\right. \bv{Z}(t)}. \nonumber
\end{eqnarray}
Lemma \ref{lemma:drift-ineq} then follows by rearranging the terms, and using the definitions of $R_{s}^{\script{T}}(t)$, $R_{m}^{\script{T}}(t)$ and $R_{m}^{D_i}(t)$ in equations (\ref{eq:input-server-rate}), (\ref{eq:income-rate-split-1}) and (\ref{eq:income-rate-split-2}). 
\end{proof}



\vspace{-.1in}
\section*{Appendix C -- Proof of Theorem \ref{theorem:gbp-per}}
In this section, we prove Theorem \ref{theorem:gbp-per}. We first present a lemma regarding queue stability and a theorem regarding rate allocation in a Benes network.  Then, we use the two results to carry out our analysis. 

We first have the following lemma. 
\begin{lemma}\label{lemma:stable}
Let  $Q(t)\geq0, t\in\{0, 1, ...\}$ be a queueing process with the following dynamics: 
\begin{eqnarray}
Q(t+1) = \max[Q(t)-1, 0] +R(t), 
\end{eqnarray}
Suppose (i) $0\leq R(t)\leq A_{\max}$ for all $t$, and  (ii) $\lim_{T\rightarrow\infty}\frac{1}{T}\sum_{t=0}^{T-1}R(t)\leq1-\eta$ for $0<\eta<1$ with probability $1$ (w.p.$1$). Then, $Q(t)$ is stable. $\Diamond$
\end{lemma}
\begin{proof}
See Appendix D. 
\end{proof}

To state the theorem needed for our analysis, we first define the notion of  a \emph{stabilizing rate allocation profile} for the \emph{fictitious} system. In the definition, we use $\script{L}$ to denote the set of network links in the fictitious network, and use $\mu^{\tsf{U}}_{m_1, m_2}$ and $\mu^{\tsf{L}}_{m_1, m_2}$ to denote the rates of the upper division flow traffic and the lower division flow traffic sent from node $m_1$ to node $m_2$, respectively. 
\begin{definition} (Stabilizing rate allocation profile)   
For an arrival rate vector $\bv{r}$, a stabilizing rate allocation profile $\bv{\mu}(\bv{r})=(\mu^{\tsf{U}}_{m, m'}, \, \mu^{\tsf{L}}_{m, m'}, \,\forall\, [m, m']\in\script{L})$ is a vector that satisfies the following:
\begin{eqnarray}
\hspace{-.5in}&&\sum_{d\leq2^{n-1}}r_{sd}\leq \mu^{\tsf{U}}_{s, m(s)},  \, \sum_{d>2^{n-1}}r_{sd}\leq \mu^{\tsf{L}}_{s, m(s)}, \,\forall\,s\in\script{S}, \label{eq:stable-rate-cond1}\\
\hspace{-.5in}&&\sum_{m'\in\script{M}_m}\mu^{\tsf{U}}_{m', m} \leq \mu^{\tsf{U}}_{m, m_u} + \mu^{\tsf{U}}_{m, m_l}, \,\forall\, m\in\cup_{j=1}^{n}\script{C}_j,\label{eq:stable-rate-cond2}\\
\hspace{-.5in}&& \sum_{m'\in\script{M}_m}\mu^{\tsf{L}}_{m', m} \leq \mu^{\tsf{L}}_{m, m_u} + \mu^{\tsf{L}}_{m, m_l}, \,\forall\, m\in\cup_{j=1}^{n}\script{C}_j,\label{eq:stable-rate-cond3}\\
\hspace{-.5in}&&\mu^{\tsf{U}}_{m, m'} + \mu^{\tsf{L}}_{m, m'}\leq1, \,\mu^{\tsf{U}}_{m, m'}, \mu^{\tsf{L}}_{m, m'}\geq0,\, \forall\,[m, m']\in\script{L},\label{eq:stable-rate-cond4}\\
\hspace{-.5in}&& \mu^{\tsf{L}}_{m, D_1} = 0, \, \mu^{\tsf{U}}_{m, D_2} = 0,\,\forall\,m\in\script{C}_n.\quad \Diamond \label{eq:stable-rate-cond5}
\end{eqnarray}
\end{definition}
In the above definition, if $m\in\script{C}_n$, i.e., $m$ is a partition node, then  $m_u=D_1$ and $m_l=D_2$. We now state the following theorem: 
\begin{theorem}\label{theorem:rate-split}
For every arrival rate vector $\bv{r}\in\Lambda_n$, there exists a stabilizing rate allocation profile $\bv{\mu}(\bv{r})=(\mu^{\tsf{U}}_{m, m'}, \mu^{\tsf{L}}_{m, m'}, \,\forall\,[m, m']\in\script{L})$ for the fictitious network that has the following property: 
\begin{eqnarray}
\mu_{m, m_u}^{\tsf{U}} = \mu_{m, m_l}^{\tsf{U}}, \,  \mu_{m, m_u}^{\tsf{L}} = \mu_{m, m_l}^{\tsf{L}},\,\forall\,m\in\cup_{j=1}^{n-1}\script{C}_j. \quad\Diamond\label{eq:rate-split}
\end{eqnarray}
\end{theorem}
\begin{proof}
See Appendix E. 
\end{proof}

Now we prove Theorem \ref{theorem:gbp-per}. 

\begin{proof} (Theorem \ref{theorem:gbp-per}) (\textbf{Part A-stability}) We start by  proving network stability. Our proof has two parts. In part one, we show that the fictitious network is stable, which implies that the nodes in columns $1$ to $n-1$ of the physical network are stable. Then, we show that each individual node in columns $n$ to $2n-1$ of  the physical network is stable.  

$\bullet$ (Fictitious network) From the auxiliary variable selection step (\ref{eq:gbp-gamma}) and the fact that the maximum first derivative of the utility functions is $\beta$, we see that whenever $H_{sd}(t)>V\beta$, \tsf{G-BP} will set $\gamma_{sd}(t)=0$. Hence, using the fact that $0\leq\gamma_{sd}(t)\leq A_{\max}$ for all time, we have: 
\begin{eqnarray}
0\leq H_{sd}(t)\leq V\beta+A_{\max}, \label{eq:h-q-bdd}
\end{eqnarray}
for all $(s, d)$ flows and for all time. 

Now consider the admission control step. We see that 
whenever $Q_{s}^{\tsf{U}}(t)>H_{sd}(t)$, $R_{sd}(t)=0$ for any upper division $(s, d)$ flows. Similarly, whenever $Q_{s}^{\tsf{L}}(t)>H_{sd}(t)$, $R_{sd}(t)=0$ for any lower division $(s, d)$ flows. Since for both $Q_{s}^{\tsf{U}}(t)$ and $Q_{s}^{\tsf{L}}(t)$, there can be at most $2^{n-1}A_{\max}$ new packet arrivals in a single time slot, 
we have for every $s\in\script{S}$ that: 
\begin{eqnarray}
Q_{s}^{\tsf{U}}(t)\leq V\beta +(2^{n-1}+1)A_{\max}, \label{eq:src-up-q-bdd}\\
Q_{s}^{\tsf{L}}(t)\leq V\beta +(2^{n-1}+1)A_{\max}. \label{eq:src-low-q-bdd}
\end{eqnarray}
Similarly, we also see that for every $(s, d)$ flow, if $q_d(t)>H_{sd}(t)$, then $R_{sd}(t)=0$. This together with (\ref{eq:h-q-bdd}) imply that: 
\begin{eqnarray}
q_d(t)\leq V\beta +(2^n+1)A_{\max}.  
\end{eqnarray}
Here the term $2^nA_{\max}$ is because in any time slot, there can be at most $2^nA_{\max}$ new packets entering $q_d(t)$. Hence, all the regulation queues are  also stable. 

Now consider a node $m\in\script{C}_1$ and look at its upper division queues $Q_m^{\tsf{UU}}(t)$ and $Q_m^{\tsf{UL}}(t)$. According to the routing and scheduling rules, in order for any of the two queues  to receive  new  arrivals, there must exist a node $s\in\script{M}_m$ such that:  
$Q_{m}^{\tsf{UU}}(t) + Q_{m}^{\tsf{UL}}(t) < 2 Q_{s}^{\tsf{U}}(t)$. 
These together with (\ref{eq:src-up-q-bdd}) and (\ref{eq:src-low-q-bdd}) imply that:  
\begin{eqnarray*}
&&Q_m^{\tsf{UU}}(t) + Q_m^{\tsf{UL}}(t) \\
&&\qquad\qquad\leq 2(V\beta+(2^{n-1}+1)A_{\max}) +2,\,\forall\,m\in\script{C}_1. 
\end{eqnarray*}
Here the last fudge factor $2$ is because at any time $t$, there can be at most $2$ new packet arrivals to node $m$. Similarly, we have for the lower division queues that: 
\begin{eqnarray*}
&& Q_m^{\tsf{LU}}(t) + Q_m^{\tsf{LL}}(t) \\
&& \qquad\qquad\leq 2(V\beta+(2^{n-1}+1)A_{\max}) +2, \,\forall\,m\in\script{C}_1.  
\end{eqnarray*}
With the above reasoning, one can show that for $m\in\script{C}_2$ , 
\begin{eqnarray}
Q_m^{\tsf{UU}}(t) + Q_m^{\tsf{UL}}(t) \leq 2^2(V\beta+(2^{n-1}+1)A_{\max})+2^2+2,\\
Q_m^{\tsf{LU}}(t) + Q_m^{\tsf{LL}}(t) \leq 2^2(V\beta+(2^{n-1}+1)A_{\max})+2^2+2. 
\end{eqnarray}
More generally, for every node $m\in\cup_{j=1}^{n-1}\script{C}_j$, we have: 
\begin{eqnarray*}
Q_m^{\tsf{UU}}(t) + Q_m^{\tsf{UL}}(t) \leq 2^{j_m}(V\beta+(2^{n-1}+1)A_{\max})+\sum_{l=1}^{j_m}2^l,\\
Q_m^{\tsf{LU}}(t) + Q_m^{\tsf{LL}}(t) \leq 2^{j_m}(V\beta+(2^{n-1}+1)A_{\max})+\sum_{l=1}^{j_m}2^l, 
\end{eqnarray*}
and for $m\in\script{C}_n$, we have: 
\begin{eqnarray*}
Q_m^{D_1}(t) \leq 2^{n}(V\beta+(2^{n-1}+1)A_{\max})+\sum_{l=1}^{n}2^l,\\
Q_m^{D_2}(t)  \leq 2^{n}(V\beta+(2^{n-1}+1)A_{\max})+\sum_{l=1}^{n}2^l. 
\end{eqnarray*}
This proves that the fictitious network and the nodes in columns $1$ to $n-1$ in the physical network are stable. 

$\bullet$ (Second half of the physical network) Now we show that the nodes in columns $n$ to $2n-1$ of the physical network are stable. Recall that in this second half of the physical network, there are only two queues at each switch module $m$, i.e., $Q_m^a(t)$ for the upper outgoing link $a$ and $Q_m^b(t)$ for the lower outgoing link $b$.

We first consider a partition node $m\in\script{C}_n$. Since the fictitious network is stable,  Lemma \ref{lemma:balanced-load} shows that for every $(s, d)$ flow, its rate is equally split among the $2^{n-1}$ partition nodes. 
Using Lemma \ref{lemma:unique-path}, we see that the total  flow rate going through the upper output link $a$ of $m$ is given by:  
\begin{eqnarray}
\sum_{d\leq2^{n-1}}\sum_{s}r_{sd}/2^{n-1}\leq 2^{n-1}(1-\eta)/2^{n-1}\leq 1-\eta. 
\end{eqnarray}
Here the first inequality uses the fact that the regulation queues are stable, which implies $\sum_{s}r_{sd}\leq1-\eta$. 
Thus, the total input rate into $Q_m^a(t)$ is no more than $1-\eta$, whereas the total output rate is $1$ according to \tsf{G-BP}. This, together with Lemma \ref{lemma:stable} and the fact that the maximum number of packets that can enter $Q_m^a(t)$ at any time is $2$,   imply that for any partition node $m\in\script{C}_n$, $Q_m^a(t)$ is stable. Similarly, one can show that $Q_m^b(t)$ is stable. 

Now we look at a node $m\in\script{C}_{n+1}$. Note that node $m$ is connected by two partition nodes in $\script{C}_n$. 
Using Lemma \ref{lemma:unique-path} and Lemma \ref{lemma:balanced-load}, we see that the total rate going through the output link $a$ of $m$ is given by: 
\begin{eqnarray*}
&&\quad2\sum_{d\in\script{O}_m^a}\sum_{s}r_{sd}/2^{n-1} \\
&&= 2\sum_{d\in[\kappa_m2^{n-1}, (\kappa_m+\frac{1}{2})2^{n-1}]}\sum_{s}r_{sd}/2^{n-1}\\
&&\leq 2^{n-1}(1-\eta)/2^{n-1}\\
&&\leq 1-\eta.
\end{eqnarray*}
A similar argument will show that the total rate going through the output link $b$ is also no more than $1-\eta$, proving that the nodes in $\script{C}_{n+1}$ are all stable. 
Now by repeatedly applying Lemma \ref{lemma:unique-path}, Lemma \ref{lemma:balanced-load}, and the above reasoning, one can show that for any node $m\in\cup_{j=n}^{2n-1}\script{C}_j$, the total input rates into $Q_m^a(t)$ and $Q_m^b(t)$  are both no more than $1-\eta$ while the service rates are both $1$. Hence, every node in the second half of the physical network is stable. This completes the proof of network stability.

(\textbf{Part B-utility}) We now prove the flow utility performance (\ref{eq:utility-gbp}). The analysis is done by  first constructing a near-optimal solution to an optimization problem that captures the optimal utility. Then, we show that our algorithm achieves a similar utility performance by comparing the Lyapunov drift values. 

To start, we use $\bv{A}=(A_{sd}, \forall\, (s,d))$ to denote the random arrival vector and use $\{\bv{R}^{(\bv{A}, k)}=(R^{(\bv{A}, k)}_{sd}, \forall\, (s,d)), k=1, 2, ...\}$ to denote a sequence of admission vectors under  arrival vector $\bv{A}$. 
We then formulate the following optimization problem: 
\begin{eqnarray}
\hspace{-.4in}&&\max:\,\,\,  \phi_{\eta}\triangleq\sum_{sd}U_{sd}(\gamma_{sd})\label{eq:utility-eta}\\
\hspace{-.4in}&&\quad\text{s.t.} \quad \gamma_{sd}\leq r_{sd}\triangleq\expectm{\sum_{k}p^{(\bv{A})}_kR^{(\bv{A}, k)}_{sd}},\,\,\forall\,\,s, d,\label{eq:cond-1}\\
\hspace{-.4in}&&\qquad\quad\sum_dr_{sd}=\sum_d\expectm{\sum_{k}p^{(\bv{A})}_kR^{(\bv{A}, k)}_{sd}}\leq 1,\,\,\forall\,\,s,\label{eq:cond-2}\\
\hspace{-.4in}&&\qquad\quad\sum_sr_{sd} = \sum_s\expectm{\sum_{k}p^{(\bv{A})}_kR^{(\bv{A}, k)}_{sd}}\leq 1-\eta,\,\,\forall\,\,d,\label{eq:cond-3}\\
\hspace{-.4in}&&\qquad\quad 0\leq R_{sd}^{(\bv{A}, k)}\leq A_{sd},\,\,\forall\,s,d,  \bv{A}, k,\label{eq:cond-4}\\
\hspace{-.4in}&&\qquad\quad\sum_{k}p^{(\bv{A})}_k=1, p^{(\bv{A})}_k\geq0, \,\,\forall\, \bv{A}. 
\end{eqnarray}
Here $p^{(\bv{A})}_k$ can be interpreted as the fraction of time the system uses the vector $\bv{R}^{(\bv{A}, k)}$ when the arrival vector is $\bv{A}$, and the expectation is taken over the random arrival vector $\bv{A}$. 

For any given $\eta$ value, denote $\bv{\gamma}^*(\eta), \bv{r}^*(\eta)$, and $\{\bv{R}^{(\bv{A}, k)*}(\eta), p^{(\bv{A})*}_k(\eta)\}_{k=1}^{\infty}$ an optimal solution of (\ref{eq:utility-eta}) and let the optimal value be $\phi_\eta^*$. Since each utility function $U_{sd}(\cdot)$ is concave increasing, we see that $\gamma_{sd}^*(\eta) = r_{sd}^*(\eta)$ for all $(s, d)$. We also see that $\bv{r}^*(\eta)\in\Lambda_n$, because (\ref{eq:cond-2}) and (\ref{eq:cond-3}) are sufficient conditions to guarantee that an arrival rate vector  is in $\Lambda_n$. 
%
Moreover, using an argument based on Caratheodory's theorem as in \cite{neelynowbook}, one can show that $\phi_0^*$, i.e., the value of (\ref{eq:utility-eta}) at $\eta=0$, provides an upper bound of the optimal utility of our problem, i.e., 
\begin{eqnarray}
\phi^*_0\geq U(\bv{r}^{\tsf{opt}}). 
\end{eqnarray}
This is so because any feasible rate solution to our problem must also satisfies all the constraints (\ref{eq:cond-1}) - (\ref{eq:cond-4}) with $\eta=0$. 

We create a  solution $\tilde{\bv{\gamma}}(\eta)$, $\tilde{\bv{r}}(\eta)$, $\{\tilde{\bv{R}}^{(\bv{A}, k)}(\eta), \tilde{p}^{(\bv{A})*}_k(\eta)\}_{k=1}^{\infty}$ for (\ref{eq:utility-eta}) as follows: 
\begin{eqnarray}
\hspace{-.3in}&&\tilde{\gamma}_{sd}(\eta)=(1-\eta)\gamma_{sd}^*(0), \,\,\tilde{r}_{sd}(\eta)=(1-\eta)r_{sd}^*(0), \\ 
\hspace{-.3in}&&\tilde{\bv{R}}^{(\bv{A}, k)}(\eta)=(1-\eta)\bv{R}^{(\bv{A}, k)*}(0), \tilde{p}^{(\bv{A})*}_k(\eta)=p^{(\bv{A})*}_k(\eta). 
\end{eqnarray}
It can be verified that  ($\tilde{\bv{\gamma}}_{\eta}$, $\tilde{\bv{r}}(\eta)$, $\{\tilde{\bv{R}}^{(\bv{A}, k)}(\eta), \tilde{p}^{(\bv{A})*}_k(\eta)\}_{k=1}^{\infty}$) is a feasible solution for (\ref{eq:utility-eta}). Denote the value of $\phi_\eta$ under this solution as $\tilde{\phi}_{\eta}$. 
Using the definition of $\beta$, we see that: 
\begin{eqnarray}
U_{sd}(r_{sd}^*(0))\leq U_{sd}(\tilde{r}_{sd}(\eta)) + \beta\eta r_{sd}^{*}(0). 
\end{eqnarray}
Therefore, we have: 
\begin{eqnarray}
U(\bv{r}^{\tsf{opt}})\leq\phi^*_0&=&\sum_{sd}U_{sd}(r_{sd}^{*}(0))\nonumber\\
&\leq& \sum_{sd}U_{sd}(\tilde{r}_{sd}(\eta)) + \beta\eta \sum_{sd}r_{sd}^{*}(0)\nonumber\\
&=&\tilde{\phi}_\eta + \beta\eta \sum_{sd}r_{sd}^{*}(0)\nonumber\\
&\leq& \phi^*_{\eta} +\eta \beta2^{n}. \label{eq:utility-bigger}
\end{eqnarray} 
Here the last step follows because $\bv{r}^*(0)\in\Lambda_n$, which implies $\sum_{sd}r_{sd}^{*}(0)\leq 2^n$. (\ref{eq:utility-bigger}) then implies that: 
\begin{eqnarray}
\phi^*_{\eta}  \geq U(\bv{r}^{\tsf{opt}}) -\beta \eta 2^{n}. \label{eq:utility-order}
\end{eqnarray}
Since $\bv{r}^*(\eta)\in\Lambda_n$,  by Theorem \ref{theorem:rate-split}, there exists a stabilizing rate allocation vector $\bv{\mu}(\bv{r}^*(\eta))$ that satisfies (\ref{eq:rate-split}) for all nodes in $\script{C}_1$ to $\script{C}_{n-1}$, which further  implies that there exists a stationary and randomized routing and scheduling policy $\Pi$ that achieves the following for all $m\in\cup_{j=1}^{n-1}\script{C}_j$ \cite{neelynowbook}: 
\begin{eqnarray*}
\hspace{-.3in}\expectm{\mu_{s, m(s)}^{\tsf{U}}(t)} &=& \mu_{s, m(s)}^{\tsf{U}}(\bv{r}^*(\eta)), \\
\hspace{-.3in}\expectm{\mu_{s, m(s)}^{\tsf{L}}(t)} &= &\mu_{s, m(s)}^{\tsf{L}}(\bv{r}^*(\eta)),\\
\hspace{-.3in}\expectm{\mu_{m, m_u}^{\tsf{UU}}(t)}& =& \mu_{m, m_u}^{\tsf{U}}(\bv{r}^*(\eta)), \\
\hspace{-.3in}\expectm{\mu_{m, m_l}^{\tsf{UL}}(t)} &=& \mu_{m, m_l}^{\tsf{U}}(\bv{r}^*(\eta)), \\
\hspace{-.3in}\expectm{\mu_{m, m_u}^{\tsf{LU}}(t)} &=& \mu_{m, m_u}^{\tsf{L}}(\bv{r}^*(\eta)), \\
\hspace{-.3in}\expectm{\mu_{m, m_l}^{\tsf{LL}}(t)} &=& \mu_{m, m_l}^{\tsf{L}}(\bv{r}^*(\eta)). 
\end{eqnarray*}
Here we again assume that if $m\in\script{C}_n$, then $m_u=D_1$, $m_l=D_2$, $\mu_{m, m_u}^{\tsf{UU}}(t)=\mu_{m, D_1}(t)$, $\mu_{m, m_l}^{\tsf{LU}}(t)=\mu_{m, D_2}(t)$, $\mu_{m, m_u}^{\tsf{UL}}(t)=0$, and $\mu_{m, m_l}^{\tsf{LL}}(t)=0$. 

Now since the \tsf{G-BP} algorithm is constructed by choosing the actions to minimize the RHS of (\ref{eq:drift-utility}), or equivalently (\ref{eq:drift-utility-rate-form}), we see that (\ref{eq:drift-utility-rate-form}) remains true if we plug in any alternate control actions. Thus, we plug in the solution ($\bv{\gamma}^*(\eta), \bv{r}^*(\eta)$, $\{\bv{R}^{(\bv{A}, k)*}(\eta), p^{(\bv{A})*}_k(\eta)\}_{k=1}^{\infty}$), and the routing and scheduling policy $\Pi$ above, which guarantees: 
\begin{eqnarray}
\hspace{-.4in}&&\expectm{\mu_{s, m(s)}^{\script{T}}(t)   - R_s^{\script{T}}(t)}\geq0,\,\,\forall\, s, \script{T}\in\Omega_s, \label{eq:service-arrival-1}\\
\hspace{-.4in}&&\expectm{\mu_{m, m(\script{T})}^{\script{T}}(t)   - R_m^{\script{T}}(t)}\geq0,\,\,\forall\, m\in\cup_{j=1}^{n-1}\script{C}_j, \script{T}\in\Omega_B,  \label{eq:service-arrival-2}\\
\hspace{-.4in}&&\expectm{\mu_{m, D_i} - R_m^{D_i}}\geq0,\,\,\forall\, m\in\script{C}_n, i=1,2,\label{eq:service-arrival-3}\\
\hspace{-.4in}&&\expectm{R_{sd}(t)-\gamma_{sd}(t)}\geq0, \expectm{1-\eta-\sum_sR_{sd}(t)}\geq0.  \label{eq:service-arrival-4}
\end{eqnarray}
Thus, using the definition of $\bv{\gamma}^*(\eta), \bv{r}^*(\eta)$, and (\ref{eq:service-arrival-1})-(\ref{eq:service-arrival-4}),  we see that after plugging in the alternative actions,  (\ref{eq:drift-utility-rate-form}) becomes: 
\begin{eqnarray}
\hspace{-.3in}&&\Delta(t) - V\expect{\sum_{s, d}U_{sd}(\gamma^{\tsf{G-BP}}_{sd}(t))\left.|\right.\bv{Z}(t)}\nonumber \\
\hspace{-.3in}&&\qquad\qquad\qquad\qquad\leq B- V\sum_{s,d}U_{sd}(\gamma_{sd}^*(\eta))\nonumber\\
\hspace{-.3in} &&\qquad\qquad\qquad\qquad\leq B-VU(\bv{r}^{\tsf{opt}}) + V\eta \beta2^{n}. \label{eq:drift-utility-bdd}
\end{eqnarray}
Here the last step follows from (\ref{eq:utility-order}). 
Taking expectations over $\bv{Z}(t)$ on both sides, summing (\ref{eq:drift-utility-bdd}) over $t=0, ..., T-1$, and rearranging terms, we have: 
\begin{eqnarray*}
\hspace{-.3in}&&TVU(\bv{r}^{\tsf{opt}}) -T V\eta \beta2^{n} -BT \\
\hspace{-.3in}&& \qquad\qquad\leq V\sum_{t=0}^{T-1}\expect{\sum_{s, d}U_{sd}(\gamma^{\tsf{G-BP}}_{sd}(t))} + \expect{L(0)}. 
\end{eqnarray*}
Dividing both sides by $TV$,  we get: 
\begin{eqnarray*}
\hspace{-.3in}&&U(\bv{r}^{\tsf{opt}}) -B/V - \eta \beta2^{n}  \\
\hspace{-.3in}&& \qquad\qquad\leq \frac{1}{T}\sum_{t=0}^{T-1}\expect{\sum_{s, d}U_{sd}(\gamma^{\tsf{G-BP}}_{sd}(t))} + \expect{L(0)}/TV. 
\end{eqnarray*}
Using Jensen's inequality, we have:
\begin{eqnarray*}
\hspace{-.3in}&&U(\bv{r}^{\tsf{opt}}) -B/V - \eta \beta2^{n}  \\
\hspace{-.3in}&& \qquad\qquad\leq\sum_{s, d}U_{sd}( \frac{1}{T}\sum_{t=0}^{T-1}\expectm{\gamma^{\tsf{G-BP}}_{sd}(t)}) + \expect{L(0)}/TV. 
\end{eqnarray*}
Taking a limit as $T\rightarrow\infty$ and using the fact that $\expectm{L(0)}<\infty$, we get: 
\begin{eqnarray}
U(\bv{r}^{\tsf{opt}}) -B/V - \eta \beta2^{n}  \leq \sum_{s, d}U_{sd}(\overline{\gamma}^{\tsf{G-BP}}_{sd}). 
\end{eqnarray}
Here $\overline{\gamma}^{\tsf{G-BP}}_{sd}$ is the average value of $\gamma_{sd}(t)$ under \tsf{G-BP}, i.e., 
\begin{eqnarray}
\overline{\gamma}^{\tsf{G-BP}}_{sd}\triangleq\lim_{T\rightarrow\infty} \frac{1}{T}\sum_{t=0}^{T-1}\expectm{\gamma^{\tsf{G-BP}}_{sd}(t)}. \nonumber
\end{eqnarray}
Finally, recall that all the admission queues $H_{sd}(t)$ are stable, which implies $\overline{\gamma}^{\tsf{G-BP}}_{sd}\leq r^{\tsf{G-BP}}_{sd}$. Therefore, 
\begin{eqnarray}
U(\bv{r}^{\tsf{opt}}) -B/V - \eta \beta2^{n}  \leq \sum_{s, d}U_{sd}(r^{\tsf{G-BP}}_{sd}). 
\end{eqnarray}
This completes the proof of the theorem. 
\end{proof}

\vspace{-.1in}
\section*{Appendix D -- Proof of Lemma \ref{lemma:stable}}
We first prove Lemma \ref{lemma:stable}.  
\begin{proof} (Lemma \ref{lemma:stable}) 
We prove the lemma by contradiction. Suppose the conclusion is not true. Then, for any finite constant $M$, there exists a time $t$ such that $Q(t)>M$. 

Since $\lim_{T\rightarrow\infty}\frac{1}{T}\sum_{t=0}^{T-1}R(t)\leq1-\eta$ with probability one, we see that for  any finite starting time $t_0$ and for any $\epsilon>0$, there exists a time $T_{(\epsilon)}<\infty$ such that for any $T\geq T_{(\epsilon)}$,  
\begin{eqnarray}
\frac{1}{T}\sum_{t=t_0}^{t_0+T-1}R(t)\leq1-\eta+\epsilon, \,\,\text{w.p.1}. \label{eq:rate-converge}
\end{eqnarray}

Now fix $\epsilon=\eta/2$ and choose $M=A_{\max}T_{(\eta/2)}$. Let $t^*$ be the time when $Q(t^*)>M$ and let $t^*_0$ be the beginning of the busy period during which the event $\{Q(t^*)>M\}$ happens, i.e., $Q(t_0^*-1)=0$, and for any time $t\in[t^*_0, t^*]$,  $Q(t)>0$. We see that: 
\begin{eqnarray*}
&&\quad\,\,\, Q(t) = \sum_{\tau=t^*_0}^{t-1}R(\tau) - (t-t^*_0-1), \\
&&\Rightarrow Q(t^*) = \sum_{\tau=t^*_0}^{t^*-1}R(\tau) - (t^*-t^*_0-1). 
\end{eqnarray*}
Since $M=A_{\max}T_{(\eta/2)}$, we must have $t^*-t^*_0-1\geq T_{(\eta/2)}$, for otherwise $\sum_{\tau=t^*_0}^{t^*-1}R(\tau)\leq M$. In this case, using (\ref{eq:rate-converge}) and $\epsilon=\eta/2$, we have with probability $1$ that: 
 \begin{eqnarray*}
&&Q(t^*) = \sum_{\tau=t^*_0}^{t^*-1}R(\tau) - (t^*-t^*_0-1)\\
&&\qquad\,\,\,\,\leq (1-\eta/2) (t^*-t^*_0-1)  - (t^*-t^*_0-1)\leq0. 
\end{eqnarray*}
This contradicts the fact that $Q(t^*)>M$. Hence,  $Q(t)\leq M$ with probability $1$ and  $Q(t)$ is stable. 
\end{proof}

\vspace{-.1in}
\section*{Appendix E -- Proof of Theorem \ref{theorem:rate-split}}
Now  we prove Theorem \ref{theorem:rate-split}. 
\begin{proof} (Theorem \ref{theorem:rate-split}) 
We use induction to prove the theorem. The idea is to construct a feasible rate allocation profile that balances the input and output rates for each switch module in the fictitious network.

We first show that the result holds for a $4\times4$ Benes network. In this case, the fictitious network is shown in Fig. \ref{fig:benes-4x4-fic}. 
\begin{figure}[cht]
  \centering
  \vspace{-.1in}
\includegraphics[height=0.8in, width=2.0in]{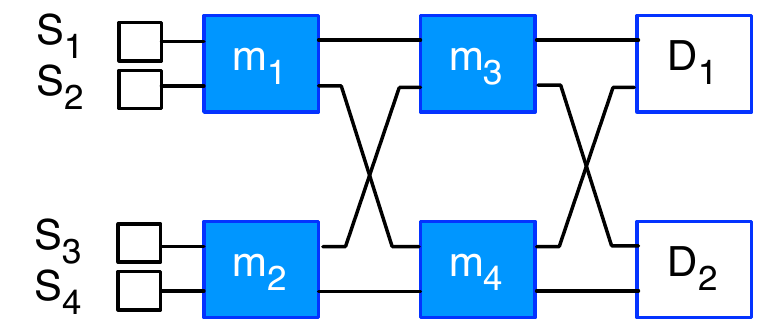}
\vspace{-.1in}
  \caption{The fictitious network for a $4\times4$ Benes network.}
  \label{fig:benes-4x4-fic}
  \vspace{-.1in}
\end{figure}

Suppose the arrival rate vector is $\bv{r}=(r_{sd}, \,\forall\,s, d)\in\Lambda_2=\{\bv{r}\,|\,\sum_{s=1}^4r_{sd}\leq1,\, \sum_{d=1}^4r_{sd}\leq1,\,\,\forall\,s, d\}$. We construct a stabilizing rate allocation profile $\bv{\mu}(\bv{r})$ as follows: 
\begin{eqnarray*}
\hspace{-.3in}&&\mu_{s_im_1}^{\tsf{U}}=\sum_{d=1,2}r_{s_id},\,\mu_{s_im_1}^{\tsf{L}}=\sum_{d=3,4}r_{s_id},\,\,\forall\,s_i=1,2,\\
\hspace{-.3in}&&\mu_{s_im_2}^{\tsf{U}}=\sum_{d=1,2}r_{s_id},\,\mu_{s_im_2}^{\tsf{L}}=\sum_{d=3,4}r_{s_id},\,\,\forall\,s_i=3,4,\\
\hspace{-.3in}&&\mu_{m_1, m_3}^{\tsf{U}}=\mu_{m_1, m_4}^{\tsf{U}}=\frac{1}{2}\sum_{d=1,2}(r_{1d}+r_{2d}),\\
\hspace{-.3in}&& \mu_{m_1, m_3}^{\tsf{L}}=\mu_{m_1, m_4}^{\tsf{L}}=\frac{1}{2}\sum_{d=3,4}(r_{1d}+r_{2d}),\\
\hspace{-.3in}&&\mu_{m_2, m_3}^{\tsf{U}}=\mu_{m_2, m_4}^{\tsf{U}}=\frac{1}{2}\sum_{d=1,2}(r_{3d}+r_{4d}),\\
\hspace{-.3in}&& \mu_{m_2, m_3}^{\tsf{L}}=\mu_{m_2, m_4}^{\tsf{L}}=\frac{1}{2}\sum_{d=3,4}(r_{3d}+r_{4d}), \\
\hspace{-.3in}&&\mu_{m_3, D_1}^{\tsf{U}}=\mu_{m_4, D_1}^{\tsf{U}}=\frac{1}{2}\sum_{d=1,2}\sum_sr_{sd},\\
\hspace{-.3in}&& \mu_{m_3, D_2}^{\tsf{L}}=\mu_{m_4, D_2}^{\tsf{L}}=\frac{1}{2}\sum_{d=3,4}\sum_sr_{sd}. 
\end{eqnarray*}
Since $\bv{r}\in\Lambda_2$, it can be verified that  $\bv{\mu}(\bv{r})$ satisfies all the constraints (\ref{eq:stable-rate-cond1}) - (\ref{eq:stable-rate-cond4}). Hence, it is a stabilizing rate allocation profile. This proves the $4\times4$ case. 

Now suppose the result holds for the $2^{n-1}\times2^{n-1}$ Benes network, we show that it also holds for the  $2^n\times2^n$ Benes network $\mathbb{B}_n$. To do so, let $\bv{r}\in\Lambda_n$ denote the input vector to $\mathbb{B}_n$ and we construct a stabilizing rate allocation profile as follows. 


First, for each input server $s$, we let 
\begin{eqnarray}
\mu_{s, m(s)}^{\tsf{U}}=\sum_{d\leq2^{n-1}}r_{sd}, \,\mu_{s, m(s)}^{\tsf{L}}=\sum_{d>2^{n-1}}r_{sd}. \label{eq:server-rate-allocation}
\end{eqnarray}
Then, for each $m\in\script{C}_1$, we let: 
\begin{eqnarray}
\hspace{-.3in}&&\mu^{\tsf{U}}_{m, m_u}=\mu^{\tsf{U}}_{m, m_l}=\frac{1}{2}\sum_{s\in\{2i_m-1, 2i_m\}}\sum_{d\leq2^{n-1}}r_{sd}, \label{eq:rate-split-c1-1}\\
\hspace{-.3in}&&\mu^{\tsf{L}}_{m, m_u}=\mu^{\tsf{L}}_{m, m_l}=\frac{1}{2}\sum_{s\in\{2i_m-1, 2i_m\}}\sum_{d>2^{n-1}}r_{sd}. \label{eq:rate-split-c1-2}
\end{eqnarray}
Here $s=2i_m-1$ and $2i_m$ are the input servers that connect to $m$. 
Note that (\ref{eq:rate-split-c1-1}) and (\ref{eq:rate-split-c1-2}) can also be viewed as equally splitting the traffic of each flow going through $m\in\script{C}_1$ to its two next-hop nodes $m_u$ and $m_l$, one in the upper subnetwork and the other in the lower subnetwork.  We thus take these as the traffic input rates to the two subnetworks of the Benes network. 

Now consider the upper subnetwork and label all the input and output ports of the upper subnetwork by $s'\in\{1,..., 2^{n-1}\}$ and $d'\in\{1, ..., 2^{n-1}\}$. 
According to the construction rules of $\mathbb{B}_n$ in Section \ref{section:benes-connect}, an input port $s'$ is connected by the switch module in row $s'$  in $\script{C}_1$ of $\mathbb{B}_n$; while an outport $d'$ connects to the switch module in row $d'$ in $\script{C}_{2n-1}$ of $\mathbb{B}_n$. 
These imply that the traffic going from input port $s'$ to output port $d'$ in the upper $2^{n-1}\times2^{n-1}$ subnetwork indeed consists of the traffic going from input ports $2s'-1$ and $2s'$ to $2d'-1$ and $2d'$ in $\mathbb{B}_n$. Denote the rate of this traffic by $\hat{r}_{s'd'}$. 
Using (\ref{eq:rate-split-c1-1}) and (\ref{eq:rate-split-c1-2}), we have: 
\begin{eqnarray}
&&\hat{r}_{s'd'}=\frac{1}{2}\bigg[r_{(2s'-1)(2d'-1)}+r_{(2s'-1)(2d')}\label{eq:subnetwork-rate-eq}\\
&&\qquad\qquad\qquad+r_{(2s')(2d'-1)}+r_{(2s')(2d')}\bigg]. \nonumber
\end{eqnarray}
Hence, we have: 
\begin{eqnarray}
 \sum_{s'}\hat{r}_{s'd'} =\frac{1}{2}\sum_{s}(r_{s(2d'-1)}+r_{s(2d')})\leq1.\label{eq:sum-source}
\end{eqnarray}
Similarly, we have: 
\begin{eqnarray}
\sum_{d'}\hat{r}_{s'd'} = \frac{1}{2}\sum_{d}(r_{(2s'-1)d}+r_{(2s')d})\leq1. \label{eq:sum-dst}
\end{eqnarray}
(\ref{eq:sum-source}) and (\ref{eq:sum-dst}) thus imply that $\hat{\bv{r}}\in\Lambda_{n-1}$. 
Hence, by induction, there exists a stabilizing rate allocation $\hat{\bv{\mu}}^{\text{up}}(\hat{\bv{r}})=(\mu_{m, m'}^{\tsf{U}, \text{up}}, \mu_{m, m'}^{\tsf{L}, \text{up}}, \,\forall\, m, m')$ that serves the arrival rate vector $\hat{\bv{r}}$ within the upper $2^{n-1}\times2^{n-1}$ subnetwork in a symmetric manner, i.e., satisfies (\ref{eq:rate-split}). Similarly, one can show that there exists a balanced stabilizing rate allocation $\hat{\bv{\mu}}^{\text{low}}(\hat{\bv{r}})=(\mu_{m, m'}^{\tsf{U}, \text{low}}, \mu_{m, m'}^{\tsf{L}, \text{low}}, \,\forall\, m, m')$ for the lower subnetwork. 

Now a stabilizing rate allocation profile for $\mathbb{B}_n$ can be constructed as follows: 
\begin{itemize}
\item For an input server $s\in\script{S}$, we use $\mu_{s, m(s)}^{\tsf{U}}$ and $\mu_{s, m(s)}^{\tsf{L}}$ as in (\ref{eq:server-rate-allocation}). 
\item For a switch module $m\in\script{C}_1$, we use $\mu^{\tsf{U}}_{m, m_u}$, $\mu^{\tsf{U}}_{m, m_l}$, $\mu^{\tsf{L}}_{m, m_u}$, and $\mu^{\tsf{L}}_{m, m_l}$ as in (\ref{eq:rate-split-c1-1}) and (\ref{eq:rate-split-c1-2}). 
\item For the switch modules in the upper subnetwork, use  $\hat{\bv{\mu}}^{\text{up}}(\hat{\bv{r}})$; for the switch modules in the lower subnetwork, use  $\hat{\bv{\mu}}^{\text{low}}(\hat{\bv{r}})$.
\end{itemize}
It can be verified that this rate vector satisfies all the constraints (\ref{eq:stable-rate-cond1}) - (\ref{eq:stable-rate-cond4}), and thus is a stabilizing rate allocation vector for $\mathbb{B}_n$. By induction, this proves the theorem. 
\end{proof}

\vspace{-.1in}

\bibliographystyle{unsrt}
\bibliography{../mybib}

\end{document}